\documentclass[a4paper,12pt]{article}
\usepackage[french, english]{babel}
\usepackage{amsfonts}
 \usepackage{color}
\usepackage{pifont}
\usepackage{latexsym,amsmath,amssymb,cite,amsthm}
\usepackage[pdfborder={0 0 0}]{hyperref}

\usepackage{color,eucal,enumerate,mathrsfs}
\usepackage[normalem]{ulem}
\usepackage{amsmath,amssymb,epsfig,bbm}

\numberwithin{equation}{section}
\theoremstyle{plain}
\newtheorem{theorem}{Theorem}[section]
\newtheorem{corollary}[theorem]{Corollary}
\newtheorem{lemma}[theorem]{Lemma}
\newtheorem{proposition}[theorem]{Proposition}
\theoremstyle{definition}

\newtheorem{definition}[theorem]{Definition}
\theoremstyle{remark}

\newtheorem{remark}[theorem]{Remark}

\newcommand{\geo}{\rm Geod}

\newcommand{\lmt}[2]{\mathop{\lim}_{{#1} \rightarrow {#2}} }

\newcommand{\lip}[1]{{\mathrm{lip}}({#1})}

\newcommand{\lmts}[2]{\mathop{\overline{\lim}}_{{#1} \rightarrow {#2}} }
\newcommand{\lmti}[2]{\mathop{\underline{\lim}}_{{#1} \rightarrow {#2}} }

\newcommand{\Ric}{{\rm{Ric}}}

\newcommand{\bRic}{{\bf Ric}}

\renewcommand{\H}{{\mathrm{Hess}}}

\newcommand{\mm}{\mathfrak m}
\newcommand{\ms}{(X,\d,\mm)}
\newcommand{\cd}{{\rm CD}(K, \infty)}

\newcommand{\rcdkn}{{\rm RCD}(K, N)}
\newcommand{\rcd}{{\rm RCD}(K, \infty)}
\newcommand{\be}{{\rm BE}(K, \infty)}

\newcommand{\vol}{{\rm Vol}_{\rm g}}
\newcommand{\g}{{\rm g}}
\newcommand{\ent}[1]{{\rm Ent}_{#1}}
\newcommand{\La}{\mathrm{L}}
\newcommand{\E}{\mathbb{E}}
\newcommand{\qq}{\mathfrak{q}}

\newcommand{\N}{\mathbb{N}}
\newcommand{\Q}{\mathfrak{V}}
\newcommand{\R}{\mathbb{R}}

\renewcommand{\P}{\mathrm{P}}

\newcommand{\V}{\mathbb{V}}
\newcommand{\dm}{\,\d \mm}

\newcommand{\supp}{\mathop{\rm supp}\nolimits}   
\newcommand{\Lip}{\mathop{\rm Lip}\nolimits}
\newcommand{\loc}{{\rm loc}}
\renewcommand{\d}{{\mathrm d}}
\newcommand{\dt}{{\d t}}
\newcommand{\ddt}{{\frac \d\dt}}

\newcommand{\D}{{\mathrm D}}
\renewcommand{\div}{{\rm div}}
\newcommand{\restr}[1]{\lower3pt\hbox{$|_{#1}$}}
\newcommand{\la}{{\langle}}
\newcommand{\ra}{{\rangle}}

\newcommand{\nchi}{{\raise.3ex\hbox{$\chi$}}}

\title{Rigidity of some functional  inequalities on RCD spaces}
\author{Bang-Xian Han\thanks{Department of Mathematics, Technion-Israel Institute of Technology, 3200003, Haifa, Israel, hanbangxian@gmail.com. The research leading to these results is part of a project that has received funding from the European Research Council (ERC) under the European Union's Horizon 2020 research and innovation programme (grant agreement No 637851). The author thanks Emanuel Milman for valuable discussions,  also  the anonymous referee for the  valuable report.}
}
\begin{document}
\date{\today}
\maketitle

\begin{center}
\begin{minipage}{130mm}{\small
\textbf{Abstract}:   We   study the cases of equality and prove  a rigidity  theorem  concerning  the 1-Bakry-\'Emery inequality.  As an application,  we   prove the rigidity and identify the extremal  functions of  the Gaussian isoperimetric  inequality,   the logarithmic Sobolev  inequality and the  Poincar\'e inequality in the setting of  $\rcd$ metric measure spaces. This  unifies and extends  to the non-smooth setting the results of  Carlen-Kerce \cite{CarlenKerce01},  Morgan \cite{Morgan05}, Bouyrie  \cite{Bouyrie17}, Ohta-Takatsu \cite{Ohta2019}, Cheng-Zhou \cite{ChengZhou17}. 
  
   Examples of  non-smooth spaces  fitting our setting are  measured-Gromov Hausdorff limits of Riemannian manifolds with uniform Ricci curvature lower  bound,  and  Alexandrov spaces with  curvature lower bound.  
Some results  including  the rigidity of the 1-Bakry-\'Emery inequality,  the rigidity of  $\Phi$-entropy inequalities are of  particular interest  even  in the  smooth setting.
}
\\
\\
{\small
\textbf{R\'esum\'e}:   Nous \'etudions les cas d'\'egalit\'e et prouvons un th\'eor\`eme de rigidit\'e sur l'in\'egalit\'e  de Bakry-\'Emery. En tant qu'application, nous prouvons la rigidit\'e et identifions les fonctions extr\^emes de l'in\'egalit\'e isop\'erim\'etrique gaussienne, de l'in\'egalit\'e de Sobolev logarithmique  et de l'in\'egalit\'e de Poincar\'e dans le cadre de espace m\'etrique mesur\'e $\rcd$. Cela unifie et \'etend aux espaces non-lisses les r\'esultats de Carlen-Kerce \cite{CarlenKerce01},  Morgan \cite{Morgan05}, Bouyrie  \cite{Bouyrie17}, Ohta-Takatsu \cite{Ohta2019}, Cheng-Zhou \cite{ChengZhou17}
  
    Des exemples d'espaces non-lisses  satisfaisant notre cadre sont les limites mesur\'e de Gromov-Hausdorff des vari\'et\'es riemanniennes de courbure  de Ricci minor\'ee, et les espaces d'Alexandrov de courbure  minor\'ee.
Certains r\'esultats, notamment la rigidit\'e de l'in\'egalit\'e de Bakry-\'Emery, la rigidit\'e des in\'egalit\'es d'entropie g\'en\'eralis\'ee sont particuli\`erement int\'eressants m\^eme dans le cadre  lisse.
}
\end{minipage}
\end{center}

\tableofcontents
\section{Introduction}\label{first section}
In this paper,  we  prove some rigidity  theorems concerning the 1-Bakry-\'Emery inequality and some other important functional inequalities on $\rcd$ metric measure spaces for positive  $K$.  Metric measure spaces satisfying  Riemannian curvature-dimension condition $\rcd$ were introduced by Ambrosio-Gigli-Savar\'e in \cite{AGS-M}, as a refinement of the  Lott-Sturm-Villani's  $\cd$ condition introduced in \cite{Lott-Villani09} and \cite{S-O1}.  Important examples of spaces satisfying $\rcd$ condition include: measured-Gromov Hausdorff limits of Riemannian manifolds with $\Ric  \geq  K $ (c.f. \cite{GMS-C}),   Alexandrov spaces with curvature $\geq K $ (c.f. \cite{ZhangZhu10}). 
We refer the readers to the survey \cite{AmbrosioICM} for an overview of this fast-growing field  and  bibliography. 

 
Let us briefly explain the primary motivation of this paper. It is  now well-known that the Bakry-\'Emery theory is an efficient tool  in the study of  geometric and functional inequalities (c.f. \cite{BE-D} and \cite{BakryGentilLedoux14}). Many important inequalities such as the  logarithmic-Sobolev inequality and the Gaussian isoperimetric inequality,  have  proofs using  heat flow or the $\Gamma_2$-calculus of  Bakry-\'Emery.  It was noticed (e.g. by Otto-Villani \cite{OttoVillani00} or Bouyrie \cite{Bouyrie17}) that the  cases of equality in the $\Gamma_2$-inequality $\Gamma_2 \geq K\Gamma$ are closely related with the rigidity of these inequalities.  More precisely, if there is {a} function attaining the equality in one of these inequalities,  there  exists  a ({\bf possibly different}) function  attaining the  equality in  the $\Gamma_2$-inequality.  For example,  when $K>0$,  any extreme function $f=f_p$ attaining the equality in the  sharp Poincar\'e inequality
\begin{equation}\label{intro:pi}
\int f^2\,\d \mm  \leq  \frac1{K} \int \Gamma(f)\,\d \mm
\end{equation}
satisfies $\Gamma_2(f_p)=K\Gamma(f_p)$, and any extreme function $f=f_l$ attaining  the equality in the sharp  logarithmic-Sobolev inequality 
\begin{equation}\label{intro:lsi}   
\int f\ln f\,\d \mm \leq \frac1{2K} \int \frac{\Gamma(f)}{f}\,\d \mm
\end{equation}
satisfies $\Gamma_2(\ln f_l)=K\Gamma(\ln f_l)$.

An interesting observation is that both $f_p, f_l$ attain  the equality in the   {\bf same}  1-Bakry-\'Emery  inequality
\begin{equation}\label{intro:bei}
\sqrt{\Gamma(P_t f)} \leq e^{-Kt} P_t \sqrt{\Gamma(f)}
\end{equation}
where $(P_t)_{t\geq 0}$ is the heat flow associated with the  Dirichlet form $\E(\cdot):=\int \Gamma(\cdot)\dm$ and the  `carr\'e du champ' $\Gamma$.
Furthermore,  both $ \div \Big(\frac{\nabla P_t f_p}{|\nabla \P_t  f_p|}\Big)$ and $ \div \Big(\frac{\nabla P_t f_l}{|\nabla \P_t  f_l|}\Big)$ attain the equalities in the $\Gamma_2$-inequality and  the 2-Bakry-\'Emery inequality.
The main aim of this paper is to understand  this observation in  general cases and  an abstract  framework.

\subsection{Bakry-\'Emery's  curvature criterion}

Let  $(M, \g, e^{-V}\vol)$ be a  weighted  Riemannian manifold equipped with  a  weighted volume measure $e^{-V}\vol$.  The canonical  diffusion operator  associated with this smooth metric measure space 
is $\La=\Delta-\nabla V$,  where $\Delta$ is the Laplace-Beltrami operator.  We say that  $(M, \g, e^{-V}\vol)$ satisfies the  $\be$ condition for some  $K\in \R$,  in the sense of  Bakry-\'Emery if
\[
\Ric_V:=\Ric+\H_V \geq K,
\]
where $\Ric$ denotes the Ricci curvature tensor and $\H_V$ denotes the Hessian of $V$.

There are several equivalent characterizations of $\be$ condition, which have their own advantages in studying  different problems. For example, the following ones are known to be equivalent to the $\be$ curvature criterion. Even in the non-smooth $\rcd$ framework,   these  characterizations  are equivalent  (in proper forms), see  \cite{AGS-M, AGS-B, Han-DCDS18, S-S} for more discussions on this topic.

\begin{itemize}
\item [{\bf a)}] {$\Gamma_2$-inequality}: $\Gamma_2(f) \geq K \Gamma(f)$  for all $f\in C_c^\infty(M)$, where   $\Gamma_2$ and $\Gamma$  are  defined by
\[
\Gamma_2(f):=\frac12 \La \Gamma(f,f) -\Gamma(f, \La f),\qquad\Gamma(f,f):=\frac12 \La (f^2)-f\La f=\g(\nabla f, \nabla f).
\]
\item [{\bf b)}]  {$p$-Bakry-\'Emery  inequality for  $p>1$}:
\begin{equation}\label{eq2-intro}
\sqrt{\Gamma (P_t f)}^p \leq e^{-pKt}P_t\big(\sqrt{\Gamma(f)}^p\big),\qquad \forall~~f\in W^{1,p}(M,  e^{-V}\vol)
\end{equation}
where $(P_t)_{t>0}$ is the semigroup generated by the diffusion operator   $\La$.

\item [{\bf c)}]  {1-Bakry-\'Emery  inequality}:
\begin{equation}\label{eq3-intro}
\sqrt{\Gamma (P_t f)} \leq e^{-Kt}P_t\big(\sqrt{\Gamma(f)}\big), \qquad \forall~~f\in W^{1, 1}(M,  e^{-V}\vol).
\end{equation}
\end{itemize}

Naturally, one would ask the following questions:  {what if  the equalities hold in these different characterizations of $\be$}? 
It will  not be surprising that  the equalities in the $\Gamma_2$-inequality, the 2-Bakry-\'Emery inequality,  and some other `second-order' inequalities, are all equivalent and  any non-constant  extreme function is affine and induces a splitting map.  For any $p>1$, by H\"older inequality, the equality in the  $p$-Bakry-\'Emery inequality yields the equality in the 1-Bakry-\'Emery inequality. Conversely, from the examples of the Poincar\'e inequality and the log-Sobolev inequality, the equality in the 1-Bakry-\'Emery inequality is {\bf strictly weaker} than  the equality in the $2$-Bakry-\'Emery inequality.  So we would ask:  what if  the equality in  the 1-Bakry-\'Emery inequality is attained by a non-constant function.  Inspired by  a recent work of Ambrosio-Bru\'e-Semola \cite{ABS19} concerning ${\rm RCD}(0, N)$ spaces, we conjecture that   on an $\rcd$ space  with $K> 0$,  the existence of a non-constant function attaining the equality in the 1-Bakry-\'Emery inequality yields the   splitting  theorem. 

\bigskip

In the first  theorem, we  prove the rigidity of the 1-Bakry-\'Emery inequality  on dimension-free  $\rcd$ spaces with $K > 0$.

\begin{theorem}[Lemma \ref{lemma}, Theorem \ref{thm:localization}, Proposition \ref{thm:berigidity2}, \ref{thm:berigidity}] \label{th3-intro}
Let  $\ms$ be an   $\rcd$ probability space  with  $K >0$.  Let  $u\in  {\rm D}(\Delta)$  be  a non-constant function with $\Delta u\in \V$.   Then the following statements are equivalent.

\begin{enumerate}
\item   ($\Gamma_2$-inequality) $\Gamma_2(u;  \varphi)=K\int \varphi\Gamma(u)\,\d \mm$ for any  $\varphi\in L^\infty$ with $\Delta \varphi \in L^\infty$;
\item  $\int (\Delta u)^2\,\d \mm=K\int \Gamma(u)\,\d \mm$;
\item  (Spectral gap) $-\Delta u=K u$;
\item  (Poincar\'e inequality)  $\int \Gamma(u)\,\d \mm=K\int u^2 \,\d \mm$;
\item  (2-Bakry-\'Emery inequality) $\Gamma(P_t u)= e^{-2Kt} P_t {\Gamma(u)}$ for some $t>0$.
\end{enumerate} 

If $u$ satisfies  one of the properties above,   it holds
\begin{enumerate}
\item  [a.] (1-Bakry-\'Emery inequality)  $\sqrt{\Gamma( P_t u)}= e^{-Kt} P_t \sqrt{\Gamma( u)}$ for all  $t>0$;
\item [b.]  $\bRic(u, u)=K\Gamma(u)\dm$; 
\item  [c.]$u$ is an affine function, this means $\H_u=0$ and $\Gamma(u)$ is a positive constant;
\item [d.] the gradient flow of $u$ induces a one-parameter semigroup of isometries  of $(X, \d)$.
\end{enumerate}

If $u$ attains the equality in the 1-Bakry-\'Emery inequality (6), we have
\begin{enumerate}
\item [e.] $\frac{\nabla P_t u}{|\nabla P_t u|}=:b$ does not depend on $t > 0$;
\item [f.] $\Delta \div(b)=-K \div(b)$, thus $\div(b)$ attains the equality in the 2-Barky-\'Emery inequality;
\item [g.]  $\nabla \div(b)=-K b$;
\item [h.]  there exists an $\rcd$ probability space $(Y, \d_Y, \mm_Y)$,  such that
the metric measure space $\ms$ is isometric to the product space $$\Big (\R, | \cdot |, \sqrt{{K}/(2\pi)} \exp(-{Kt^2}/2)\, \d t\Big) \times (Y, \d_Y, \mm_Y)$$ equipped with the $L^2$-product metric and the product measure;
\item  [i.] $ u$ can be  represented  in the coordinates of the product space $\R \times Y$ by $$ u(r, y)=\int_0^r g(s)\,\d s \qquad \forall (r, y) \in \R\times Y$$ for some non-negative $g\in L^2(\R, \sqrt{{K}/(2\pi)} \exp(-{Kt^2}/2)\, \d t)$.
In particular, if $u$ attains equality in the 2-Bakry-\'Emery inequality,   there is  a constant $C$ such that $$P_t u(r, y)=C e^{Kt}  r\qquad \forall (r, y) \in \R\times Y, \quad t>0. $$
\end{enumerate}

\end{theorem}

\begin{remark}
Concerning a $\rcd$ probability space with negative $K$, we also prove in \label{coro:negative} that the equality in the 1-Bakry-\'Emery inequality can not be attained by any non-constant function.   For spaces with infinite volume measure, it is still unknown to us. But we conjecture that a similar splitting theorem also holds for negative $K$, at least for $\rcdkn$ spaces with $N<\infty$.
\end{remark}

\subsection{Gaussian isoperimetric  inequality}
For $K>0$, let $\phi_K(t)=\sqrt{\frac{K}{2\pi}} \exp(-\frac{Kt^2}2)$ be  a Gaussian-type (probability) density function on $\R$. It is known that $(\R, |\cdot|, \phi_K\mathcal L^1)$ is a model space with synthetic Ricci curvature lower bound $K$.

 Let $\Phi_K$ denote the error function
\[
\Phi_K(t):=\int_{-\infty}^t \phi_K(s)\,\d s.
\]
It can be seen that $\Phi_K$ is continuous and strictly increasing, so its inverse $\Phi_K^{-1}$ is well-defined. We define the Gaussian isoperimetric profile $I_K:(0, 1) \mapsto [0, \sqrt{\frac{K}{2\pi}}]$ by
\begin{equation}\label{intro:GII-1}
I_K(t):= \phi_K\circ \Phi_K^{-1}(t),
\end{equation}
and we define $I_K(t)=0$ for $t=0, 1$.
It can be seen that $I_K=\sqrt K I_1$ and  $I''_K I_K=-K$. In particular, $I_K(t)$ is strictly concave in $t$ and increasing in $K$.

\bigskip

Let  $\gamma_n=\Pi_{i=1}^n\phi_1(x_i)\d x_i$ be the $n$-dimensional standard Gaussian measure on $\R^n$.  Based on an isoperimetric inequality on the discrete cube and central limit theorem,   Bobkov \cite{Bobkov97} proved the following functional version of the  Gaussian isoperimetric inequality
\begin{equation}\label{intro:eq0}
I_1\left (\int f \,\d \gamma_n \right ) \leq \int \sqrt{I_1(f)^2+|\nabla f|^2}\,\d \gamma_n
\end{equation}
for any  Lipschitz function $f$ on $(\R^n, |\cdot|, \gamma_n)$ with values in $[0, 1]$.

 In \cite{BakryLedoux96}, Bakry and Ledoux   proved  Bobkov's inequality  \eqref{intro:eq0} on smooth metric measure spaces  using a semigroup method. Recently,  by adopting the argument of Bakry-Ledoux, Ambrosio-Mondino  \cite{AmbrosioMondino-Gaussian}  obtain  Bobkov's inequality in  the non-smooth $\rcd$ setting.

\bigskip
One interesting problem is:  {when does the equality hold in Bobkov's inequality \eqref{intro:eq0}}?
In \cite[Section 2]{CarlenKerce01},  by extending ideas of Ledoux  \cite{Ledoux98},  Carlen and Kerce  characterized the cases of  equality  in \eqref{intro:eq0} for Gaussian space.    Recently,  Carlen-Kerce's technique is adopted by Bouyrie \cite{Bouyrie17} to study this problem on  weighted  Riemannian manifolds satisfying the  $\be$ condition with $K>0$.

\bigskip

In this paper,  we will study the cases of equality in  Bobkov's inequality on $\rcd$ spaces. We will identify all the  extremal functions, and prove that  any non-trivial extreme function induces an isometry map  from this space to a product space.

Let us explain how to formulate  Bobkov's inequality on an $\rcd$  metric measure space $\ms$. 
Denote by $\V$ the space of 2-Sobolev functions, defined   as the collection of functions $f\in L^2(X, \mm)$  such that there exists a sequence  $(f_n)_n \subset \Lip(X, \d)$  converging  to  $f$ in $ L^2$ and $\lip{f_n} \to G$ in $L^2$ for some $G$, where $\lip{f_n}$ is the local Lipschitz constant of $f_n$ defined by
 \[
 \lip{f_n}(x):=\lmts{y}{x} \frac{|f_n(y)-f_n(x)|}{\d(y, x)} 
\]
(and we define $\lip{f_n}(x)=0$ if $x$ is an isolated point).   It is known that there exists a minimal function in $\mm$-a.e. sense, denoted  by $|\nabla f|$,  called minimal weak upper  gradient. 
If  $(X, \d)$ is a Riemannian manifold and $\mm=
\vol$ is its volume measure, it is known that $|\nabla f|=\lip{f}$  for any $f\in \Lip$.

On $\rcd$ spaces, it is known that (c.f. \cite{AGS-C, AGS-M})  the  functional $\V\ni f \mapsto \E(f)=\int |\nabla f|^2\,\d \mm$ is lower semi-continuous (w.r.t. weak $L^2$-convergence), and it is a quasi-regular, strongly local, conservative Dirichlet form admitting a carr\'e du champ $\Gamma(f):=|\nabla f|^2$. 

 Let   $(P_t)_{t\geq 0}$ be the  $L^2$-gradient flow of $\E$ with generator $\Delta$.  If  $\ms$ is a  smooth Riemannian manifold with boundary, it is known that $(P_t)$ is  the Neumann  heat flow and $\Delta$ is the (Neumann) Laplace-Beltrami operator.  For any  $f\in  L^1$  with values in $[0, 1]$ and $K>0$,  we define $J_K(f)\in [0, +\infty]$ by 
\begin{equation}\label{intro:eq1}
J_K(f):= \lmti{t}0 \int \sqrt{I_K(P_t f)^2+|\nabla P_t f|^2}\,\d \mm.
\end{equation}

\begin{definition}[Bobkov's inequality on metric measure spaces]\label{def:bobkov}
We say that a general metric measure space  $\ms$ supports the  $K$-Bobkov's isoperimetric inequality  if  for all  measurable $f\in L^1(X, \mm)$ with values in $[0, 1]$,
\begin{equation}\label{intro:eq2}
  I_K \left (\int f \,\d \mm \right ) \leq J_K(f).
\end{equation}
\end{definition}

\begin{remark}
It is known that $\mm(X)<\infty$ if $\ms$ satisfies $\rcd$ with $K>0$ (c.f. \cite[Theorem 4.26]{S-O1}). Without loss of generality, we can assume that $\mm$ is a probability measure.   Furthermore,  the assumption `$f\in L^1(X, \mm)$' in  Definition \ref{def:bobkov} could be removed.
\end{remark}

Applying \eqref{intro:eq2} with a characteristic function $f=\nchi_E$ for a  Borel set $E \subset X$,  we get the following {\sl Gaussian isoperimetric inequality}
\begin{equation}\label{eq:gii}
P(E) \geq I_K\big(\mm(E)\big)
\end{equation}
 where  $P(E)$ is  the perimeter function  defined by  $P(E):=|\D \nchi_E|_{\rm TV}(X)$, and  $|\D \nchi_E|_{\rm TV}$ is the total variation of  $\nchi_E$ (c.f. \cite{ADM-BV, ADMG-P} for more details above BV functions and the perimeter function on metric measure spaces).

By  lower semi-continuity of weak gradients and Bakry-\'Emery's gradient estimate $|\lip{ P_t f}|^2 \leq e^{-2Kt}P_t \big( |\nabla f|^2 \big)$ (see \cite[Theorem 6.2]{AGS-M}), we can see that
$$J_K(f)=\int \sqrt{I_K(f)^2+|\nabla f|^2}\,\d \mm$$
for $f\in \Lip$.
In addition,  we can see that Bakry-\'Emery's gradient estimate yields the irreducible of $\E$, i.e. $|\nabla f|=0 $ implies  that $f$ is  constant.  Since irreducibility implies ergodicity of the heat flow (see for instance \cite[Section 3.8]{BakryGentilLedoux14}), we know $P_t f \to \int f \dm$ in $L^2$ as $t\to \infty$.  Notice that by   2-Bakry-\'Emery inequality, $ \lmt{t}{\infty} |\nabla P_t f| = 0$ in $L^2$. Thus we get  
$$\lmti{t}\infty \int \sqrt{I_K(P_t f)^2+|\nabla P_t f|^2}\,\d \mm= I_K \left (\int f \,\d \mm \right ).$$

In Proposition \ref{prop1:monotone} we prove that the function $t\mapsto J_K(P_t f)$ is non-increasing on $\rcd$ spaces with positive $K$. From the discussions above we know these spaces support Bobkov's inequality. In particular, $f$ attains the equality in  Bobkov's inequality if and only if $J_K(P_t f)$ is a constant function in $t$.
Then, in Proposition \ref{prop2:splitting} we prove the rigidity of  Bobkov's inequality, which extends    \cite[Theorem 1]{CarlenKerce01} and \cite[Theorem 1.4]{Bouyrie17} to the non-smooth setting.

\begin{theorem}[Proposition \ref{prop1:monotone} and  \ref{prop2:splitting}]\label{th1-intro}
Assume that a metric measure space $\ms$ satisfies  $\rcd$ for some $K>0$. Then $\ms$ supports $K$-Bobkov's isoperimetric inequality.

Furthermore,  $ I_K \left (\int f \,\d \mm \right ) =J_K(f)$ for some non-constant $f\in L^\infty$ if and only if
 $$\ms \cong \Big (\R, | \cdot |, \sqrt{{K}/(2\pi)} e^{-{Kt^2}/2}\mathcal \d t \Big) \times (Y, \d_Y, \mm_Y)$$ for some  $\rcd$ space $(Y, \d_Y, \mm_Y) $,
and   up to change of variables, $f$ is either  the indicator function of a half space 
\[
f(r, y)=\nchi_E, \qquad E={(-\infty, e] \times Y}, (r, y) \in \R \times Y
\]
where $e\in \R\cup \{+\infty\}$ with  $\int_{-\infty}^e \phi_K(s)\,\d s=\int f\dm$;
or else, there are  $a=(2\int f)^{-1}$ and  $b=\Phi^{-1}_K\big(f(0, y)\big)$ such that
$$f(t, y)=\Phi_K(at+b)=\int_{-\infty}^{at+b} \phi_K(s)\,\d s.$$

\end{theorem}

\subsection{$\Phi$-entropy inequalities}\label{section2:intro}

Let $\Phi$ be a continuous function defined on an interval $I \subset \R$. For any  $I$-valued function $f$,    the $\Phi$-entropy of $f$ is defined by
 \[
 \ent{\mm}^\Phi(f):= \int \Phi(f)\dm.
\]
Using a similar method as Chafa\"i \cite{Chafai04} (see also Bolley-Gentil \cite{BolleyGentil10}), 
we can  prove the following  $\Phi$-entropy inequality on $\rcd$ spaces. It can be seen that the  Poincar\'e inequality and the log-Sobolev inequality are both $\Phi$-entropy inequalities.

\begin{proposition} [Proposition \ref{prop:phi}]\label{prop:phi:intro}
Let  $\ms$ be a  metric measure space satisfying  $\rcd$ condition for some $K>0$. Let $\Phi$ be a $C^2$-continuous  strictly convex function on an interval $I\subset \R$ such that  $\frac{1}{\Phi''}$
is concave.  Then $\ms$ supports  the following $\Phi$-entropy inequality:
\begin{equation}\label{eq1-prop:phi:intro}
\ent{\mm}^\Phi(f)-\Phi\Big(\int f\dm \Big) \leq  \frac1{2K} \int \Phi''(f) \Gamma(f)\dm
\end{equation}
for any  $I$-valued function  $f$.
\end{proposition}

Furthermore, we completely characterize the cases of equality in $\Phi$-entropy inequalities. In particular, we prove that  the Poincar\'e inequality and  the log-Sobolev inequality are  essentially the only $\Phi$-entropy inequalities,  such that the corresponding  equalities can be attained by non-trivial functions.

\begin{theorem}\label{th2-intro}
Let  $\ms$ be a  metric measure space satisfying  $\rcd$ for some $K>0$. Assume there is a function  $\Phi$ which  fulfils the conditions in Proposition \ref{prop:phi:intro}, and a non-constant function $f$ attaining  the equality in  the corresponding  $\Phi$-entropy inequality. Then
\begin{enumerate}
\item $f$ attains the equality in the  1-Bakry-\'Emery inequality, so that $\ms$ is isometric to 
$$ \Big (\R, | \cdot |, \sqrt{{K}/(2\pi)} e^{-{Kt^2}/2}\mathcal \d t \Big) \times (Y, \d_Y, \mm_Y)$$ for some  $\rcd$ space $(Y, \d_Y, \mm_Y) $;
\item $\Phi'(f)$ attains the equality in  the 2-Bakry-\'Emery inequality;
\item up to affine coordinate transforms, additive and  multiplicative constants, $\Phi=x^2$ or $x\ln x$.  In these cases,  $f(r, y)$ can be written as $a_p r$ or  $e^{a_l r-a_l^2/2K}$ for  some constants $a_p, a_l\in \R$.
\end{enumerate} 
 
\end{theorem}

\begin{remark}
It is known that  Bobkov's isoperimetric inequality yields some  important inequalities (even without any curvature condition).  For example,  from  \cite[Theorem 3.2]{BakryLedoux96} we  know  $K$-Bobkov's inequality yields  the  { $K$-logarithmic Sobolev inequality}
\begin{equation}\label{eq:lsi}
\int f\ln f\,\d \mm \leq \frac1{2K} \int \frac{|\nabla f|^2}{f}\,\d \mm
\end{equation}
for any non-negative locally Lipschitz function $f$ with $\int f\,\d \mm=1$.  It is known  (c.f. Lott-Villani \cite{Lott-Villani07},  Gigli-Ledoux \cite{GigliLedoux})  that  the $K$-logarithmic Sobolev inequality implies the  { $K$-Talagrand inequality}
\begin{equation}\label{eq:ti}
W^2_2(f\mm, \mm) \leq \frac 2K \int f\ln f\,\d \mm 
\end{equation}
for any $f$ with $\int f\,\d \mm=1$. It is known (using Hamilton-Jacobi semigroup, c.f.  \cite[Theorem 1.8]{LV-HJ}  and   \cite[Section 3]{AGS-C})  that  the $K$-Talagrand inequality implies the $K$-Poincar\'e inequality (or $K$-spectral gap)
\begin{equation}\label{eq:pi}
\int f^2\,\d \mm \leq \frac1{K} \int |\nabla f|^2\,\d \mm
\end{equation}
for any  locally Lipschitz function $f$ with $\int f\,\d \mm=0$.

 Inspired by the implications of  Bobkov's inequality discussed above,  one would ask whether we can deduce   the rigidity of  the Poincar\'e inequality and the log-Sobolev inequality  (Theorem \ref{th2-intro}) from the rigidity of  Bobkov's inequality (Theorem \ref{th1-intro})   or not.  
 For example, assume there is a non-constant  function attaining  the equality in the  Poincar\'e inequality, then  $\ms$ does not support $(K+\frac 1n)$-Bobkov's inequality for any $n \in \N$. So for any $n\in N$ there is $f_n \in \Lip(X, \d) \cap L^\infty$ such that 
\begin{equation}\label{eq:inversebi}
\sqrt{\frac{K}{2\pi}} \geq  I_{K+\frac 1n} \left (\int f_n \,\d \mm \right ) > J_{K+\frac 1n}(f_n) \geq 0.
\end{equation}
Thus there is a subsequence of $(f_n)$  converging to some  $f$  in $L^2$. Letting $n \to \infty$ in \eqref{eq:inversebi}, by continuity of $(K, t) \mapsto I_K(t)$, Fatou's lemma and lower semi-continuity of $\E$, we obtain
\[
I_{K} \left (\int f \,\d \mm \right ) \geq J_{K}(f).
\]
Combining with $K$-Bobkov's inequality  we get  $I_{K} \left (\int f \,\d \mm \right )=J_K(f)$.

However, we can not assert that $f$ is not constant,  because we do not know much about $(f_n)$ except  its  existence.
\end{remark}

\begin{remark}
Concerning an extremal function $f$ of the log-Sobolev inequality, it was conjectured by Otto-Villani  \cite[Page 391]{OttoVillani00} that $\ln f$ attains the equality in  the $\Gamma_2$-inequality $\Gamma_2\geq K\Gamma$. Unfortunately, due to lack of regularity,   we can not use  second-order differentiation formula  as suggested in \cite{OttoVillani00} on curved spaces. 

Recently,  Ohta-Takatsu  \cite{Ohta2019} give a rigorous proof to the rigidity of the  log-Sobolev inequality on smooth metric measure spaces, using a localization argument which benefits from a breakthrough of Klartag \cite{KlartagNeedle}.  As mentioned in \cite[\S 4]{Ohta2019},  the rigidity of the  log-Sobolev inequality on  $\rcd$ spaces was an open problem due to lack of `needle decomposition' on dimension-free $\rcd$ spaces.

Thus the novelty of our result is that it gives an affirmative answer to the  conjecture of Otto-Villani, and extends the result of Ohta-Takatsu to $\rcd$ spaces.
\end{remark}

\subsection{Structure of the paper}
In the first part of Section \ref{section:be} we review some basic results about the non-smooth Bakry-\'Emery theory and  calculus on metric measure spaces.  Most of these results can be found in  the papers of  Ambrosio-Gigli-Savar\'e \cite{AGS-B, AGS-C,  AGS-M}, Gigli  \cite{G-N} and Savar\'e \cite{S-S}. In the second part, we study the cases of equality in the 2-Bakry-\'Emery inequality.

In Section \ref{sec:be} we prove the rigidity of the 1-Bakry-\'Emery inequality.  This  extends  the result  of Ambrosio-Bru\'e-Semola \cite{ABS19} to dimension-free $\rcd$  spaces with $K>0$.  Some important  tools used there are the continuity equation theory in the non-smooth framework developed by Ambrosio-Trevisan \cite{AT-W}, and the functional analysis tools by Gigli \cite{G-N}.  We remark that the proof in \cite{ABS19}  relies  on a  two-sides heat kernel estimate, and it seems that the proof works only for $K=0$ case. 

In Section \ref{sec:funct}, we apply the results obtained in the previous two sections to study the rigidity of  Bobkov's Gaussian isoperimetric inequality and $\Phi$-inequalities. The  arguments  in this section are not totally new,  similar  semigroup arguments  were  used by  Carlen-Kerce \cite{CarlenKerce01}, Chafa\"i \cite{Chafai04} etc. in the study of related problems on smooth metric measure spaces.
\section{Synthetic curvature-dimension conditions}\label{section:be}

\subsection{$\Gamma_2$-calculus on metric measure spaces}

\begin{definition}[Lott-Sturm-Villani's curvature-dimension condition, c.f. \cite{Lott-Villani09, S-O1}]
 We say that a metric measure space $\ms$  is $\cd$ for some $K\in \R$ if  the entropy  functional  $\ent{\mm}$ is $K$-displacement convex  on  the $L^2$-Wasserstein space $(\mathcal{P}_2(X), W_2)$. This means,  for any two probability measures $\mu_0, \mu_1 \in \mathcal {P}_2 (X)$ with $\mu_0, \mu_1 \ll \mm$, there  is  a $L^2$-Wasserstein geodesic $(\mu_t)_{t\in [0,1]}$ such that 
 \begin{equation}\label{eq1.5-intro}
\frac K2 t(1-t)W^2_2(\mu_0, \mu_1)+{\rm Ent}_\mm(\mu_t) \leq t{\rm Ent}_\mm(\mu_1)+(1-t){\rm Ent}_\mm(\mu_0)
\end{equation}
where ${\rm Ent}_\mm(\mu_t)$ is defined as $\int \rho_t \ln \rho_t\,\d \mm$ if $\mu_t=\rho_t\,  \mm$, otherwise  ${\rm Ent}_\mm(\mu_t)=+\infty$.
\end{definition}

As we introduced in the Introduction, the energy  form $\E(\cdot)$  is defined on $L^2(X,\mm)$ by 
\begin{eqnarray*}
\E(f) &:=& \inf\!\Big\lbrace \liminf_{n \to \infty} \int_X \lip{f_n}^2\d\mm : f_n \in \Lip_{b}(X),\ \! f_n \to f \text{ in } L^2(X,\mm)\Big\rbrace\\
&=&\int_X\big| \nabla  f|^2\,\d\mm
\end{eqnarray*}
where $\lip{f}(x) := \limsup_{y\to x} \vert f(x) - f(y)\vert/\!  \d(x,y)$ denotes the {local Lipschitz slope} at $x\in X$  and $|\nabla  f|$ denotes the minimal weak upper gradient. We refer the readers  to \cite{AGS-C, C-D} for details about the  theory of Sobolev space on metric measure spaces.

We say that $\ms$ is an $\rcd$ space if it is $\cd$,  and $\E(\cdot)$ is a quadratic form.  In this case, it is known that  $\E$  defines a quasi-regular, strongly local, conservative Dirichlet form admitting a carr\'e du champ $\Gamma(f):=|\nabla  f|^2$ (c.f. \cite{AGS-B} and \cite{BE-D}).
Denote $\V={\rm D}(\E)=\{f: \E(f)<\infty\}$. For any $f, g\in \V$,   by polarization, we define $$\Gamma(f, g) :=\frac14 \big ( \Gamma(f+g)-\Gamma(f-g)\big ),$$ and  $$\E(f,g)=\int \Gamma(f,g)\,\d\mm.$$

The heat flow $(P_t)$  is defined as the  gradient flow  of $\E$ in $L^2(\mm)$. It is known that  $P_t$ is linear and self-adjoint (c.f. \cite{AGS-M}).  We recall the following regularization properties of $(P_t)$, ensured by the theory of
gradient flows and maximal monotone operators.
\begin{lemma}[A priori estimates]\label{lemma:reg}
For every $f \in  L^2(\mm)$ and $t > 0$ it holds
\begin{enumerate}
\item $\|P_t f \|_{L^2} \leq  \| f\|_{L^2}$;
\item $\E(P_t f) \leq \frac 1{2t} \| f\|^2_{L^2}$;
\item $\|\Delta P_t f\|_{L^2} \leq \frac 1{t} \| f\|_{L^2}$.
\end{enumerate}
\end{lemma}

\bigskip

 Let us recall  the notion of  non-smooth vector fields  introduced by Weaver in  \cite{Weaver01} (see also \cite{AT-W} and  \cite{G-N}).
 
 \begin{definition}
 We say that a linear functional $b:  \Lip_{\rm bs} (X, \d) \mapsto L^0(X, \mm)$ is 
  an $L^2$-derivation, and write $b\in L^2(TX)$ (or $b\in L^2_\loc (TX)$ resp.),  if it satisfies the following properties.
  \begin{enumerate}
  \item  Leibniz rule: for any $f, g\in \Lip_{\rm bs}(X, \d)$ it holds
  \[
  b(f g) = b(f) g + f b(g).
  \]
  \item  $L^2$-bound: there exists $g \in  L^2(X, \mm)$ (or  $L^2_\loc (X, \mm)$ resp.) such that
  \[
  |b(f )| \leq g |\lip{f }|,\qquad\mm-\text{a.e.~ on}~ X, 
  \]
for any $f \in \Lip_{\rm bs}$ and we denote by $|b|$ the minimal (in the $\mm$-a.e. sense) $g$ satisfying such property. 
  \end{enumerate}
 \end{definition}

In   \cite{G-N} Gigli introduces  the so-called tangent and cotangent modules over  metric measure spaces,  and proves the identification results between $L^2$-derivations and elements of the tangent module $L^2(TX)$.

\begin{proposition}[Section 2.2, \cite{G-N}]
Let    $\E$   be  the Dirichlet form associated with the  metric measure space $\ms$, and let   $\Gamma$  be the  carr\'e du champ defined on $\V$. Then 
there exists  a $L^\infty$-Hilbert module  $L^2(TX)$ satisfying the following properties.
\begin{enumerate}
\item For any $f \in \V $, there is  a derivation $\nabla f  \in L^2(TX)$ defined  by  the formula 
\[
\nabla f(g) =\Gamma(f, g),\qquad\forall g\in \Lip(X, \d).
\]
\item  $L^2(TX)$ is a module over the commutative ring $L^\infty(X, \mm)$.

\item  $L^2(TX)$ is a Hilbert space equipped with the norm $\| \cdot \|$ which is compatible with the semi-norm $\E$ on $\V$, i.e.  it holds  the following  correspondence 
\[
\V \ni f \mapsto \nabla f \in L^2(TX), \quad \text{s.t.} \quad \| \nabla f\|^2=\E(f).
\]

\item The norm $\| \cdot \|$ is  induced  by a pointwise inner product $\la \cdot, \cdot \ra$ satisfying
\[
\la \nabla f, \nabla g \ra=\Gamma(f,g),\qquad\mm-\text{a.e.}
\]
and
\[
\la h \nabla f, \nabla g \ra=h\la  \nabla f, \nabla g \ra, \qquad \mm-\text{a.e.}
\]
for any $f, g\in \V$ and $h\in L^\infty_\loc$.

\item  $L^2(TX)$ is generated by $\{\nabla g: g\in \V\}$ in the following sense. For any $v\in L^2(TX)$, there exists a sequence $v_n=\sum_{i=1}^{M_n} a_{n,i} \nabla g_{n,i}$ with $a_{n,i}\in L^\infty$ and $g_{n,i}\in \V$,  such that $\| v-v_n\| \to 0$ as $n\to \infty$.
\end{enumerate}
\end{proposition}

Via integration by parts, we can define the  divergence of vector fields.

\begin{definition}
Let $b\in L^2_\loc (TX)$. We say that $b \in {\rm D}(\div)$ if there exists $g \in  L^2(X, \mm)$ such that
\[
\int \la b, \nabla f \ra\dm=\int b(f)\dm =- \int  gf \dm \qquad \text{for any }~~f \in  \Lip_{\rm bs} (X, \d). 
\]

By a density argument it is easy to check that such function $g$ is unique (when it exists) and we will denote it by $\div(b)$.
\end{definition} 
In particular, the Dirichlet form $\E$ induces a densely defined selfadjoint operator $\Delta:{\rm D}(\Delta)\subset \V \mapsto L^2$ satisfying
$\E(f,g)=-\int g\Delta f\,\d\mm$
for all $g\in \V$.

Put 
\begin{equation*}
 \Gamma_2(f; \varphi):=\frac12 \int \Gamma(f) \Delta \varphi\,\d\mm-\int \Gamma(f, \Delta f) \varphi\,\d\mm
\end{equation*}
and 
$
{\rm D}(\Gamma_2):=\Big\{ (f, \varphi): f,\varphi\in {\rm D}(\Delta), \
\Delta f\in\V, \
\varphi,\Delta\varphi \in {L^\infty}\Big\}
$.

It is proved in \cite{AGS-M} (and also \cite{AGMR-R} for $\sigma$-finite case) that $\rcd$  implies the  following non-smooth Bakry-\'Emery condition $\be$.

\begin{proposition}[The Bakry-\'Emery condition]\label{def-bevar}
Let $\ms$ be an $\rcd$ space. Then the corresponding  Dirichlet form    $\E$ satisfies the following   $\be$ condition  
\begin{equation}\label{eq-bevar}
 \Gamma_2(f; \varphi) \geq K\int \varphi \Gamma(f)\dm
\end{equation}
for all $(f, \varphi)\in {\rm D}(\Gamma_2)$ with $\varphi \geq 0$.

\end{proposition}

Under some natural regularity assumptions on the  distance canonically associated with the Dirichlet form,  the converse implication is also true, see \cite{AGS-B} for more details.

\bigskip

 We have the following crucial properties  obtained   by Savar\'e  \cite{S-S} and  Gigli \cite{G-N}.  Recall that  the space of test functions is defined as  ${\rm TestF}:= \big\{f\in {\rm D}({ \Delta})\cap L^\infty:   \Delta f\in \V, \, \Gamma(f)\in L^\infty\big\}$.  It is known that ${\rm TestF}$ is dense in $\V$ (c.f. \cite[(3.1.6)]{G-N}).
 
\begin{proposition}\label{prop: self}
Let $\ms$ be an $\rcd$ space. Then
\begin{enumerate}

\item For any $f\in {\rm TestF}$, we have $\Gamma(f) \in \V$ and 
\[
\E\Big(\Gamma(f)\Big) \leq -\int \Big (2K\Gamma(f)^2+2\Gamma(f) \Gamma(f, \Delta f) \Big)\dm.
\]
\item For every $f\in{\rm D}(\Delta)$, we have $\Gamma(f)^{1/2}\in\V$ and
$$\E\Big(\Gamma(f)^{1/2}\Big)\le \int (\Delta f)^2\,\d\mm- K\cdot \E(f).$$

\item For any $f \in {\rm D}({\Delta})$ there  is  a continuous symmetric $L^\infty$-bilinear map $\H_f(\cdot, \cdot)$ defined on $[L^2(TX)]^2$,   with values in $L^0(X, \mm)$ (c.f. \cite[Corollary 3.3.9]{G-N}).
In particular, if $f, g, h \in  {\rm TestF}$ (c.f.   \cite[Proposition 3.3.22]{G-N},  \cite[Lemma 3.2]{S-S}),  $\H_f(\cdot, \cdot)$ is given by the following formula:
\begin{equation}\label{eq:hessian}
2\H_f(\nabla g, \nabla h)=\Gamma\big(g, \Gamma(f, h)\big) +\Gamma\big(h, \Gamma(f, g)\big)-\Gamma\big(f, \Gamma(g,h)\big).
\end{equation}
\end{enumerate}
\end{proposition}

To introduce the measure-valued `Ricci tensor', we briefly recall the notion of measure-valued Laplacian ${\bf \Delta}$ (c.f. \cite{S-S, G-O}).
We say that  $f\in {\rm D}({\bf \Delta}) \subset \V$ if there exists a  signed Borel measure  $\mu=\mu_+-\mu_-\in {\rm Meas}(X)$  charging no capacity zero sets 
such that 
 \[
\int \overline{\varphi} \,\d\mu=-\int \Gamma(\varphi.
f)\,\d \mm\]
for any $\varphi\in \V$ with quasi-continuous representative
$\overline{\varphi}\in L^1(X, |\mu|)$.
If  $\mu$ is unique,  we denote it by ${\bf \Delta} f$. If ${\bf \Delta} f \ll \mm$,  we also denote its density by $\Delta f$ if there is no ambiguity.

\begin{proposition}[See  \cite{G-N}, \S 3 and   \cite{S-S}, Lemma 3.2]\label{prop:measurebochner}
Let $\ms$  be a  $\rcd$ space.  Then for any $f\in {\rm TestF}_\loc:= \big\{f\in {\rm D}({\bf \Delta})\cap L^\infty_\loc:   \Delta f\in \V_\loc, \, \Gamma(f)\in L^\infty_\loc\big\}$, it holds  $\Gamma(f) \in {\rm D}({\bf \Delta})$
 and  the following non-smooth Bochner inequality
\begin{equation*}
{\bf \Gamma}_2(f):=\frac 12 {\bf \Delta} \Gamma(f)-\Gamma(f, \Delta f)\,\mm \geq \Big{(} {K}\Gamma(f)+\|\H_f\|^2_{\rm HS}\Big{)}\,\mm.
\end{equation*}
Denote ${\rm TestV}_\loc:=\{\Sigma_ {i=1}^n a_i\nabla f_i:  n\in \N, a_i, f_i \in {\rm TestF}_\loc \}$.  There is a measure-valued symmetric bilinear map  $\bRic:   [{\rm TestV}_\loc] ^2 \mapsto   {\rm Meas}(X)$
satisfying the following properties
\begin{enumerate}
\item for any $f\in {\rm TestF}_\loc$,
\[
\bRic(\nabla f, \nabla f):=\underbrace{\frac 12  {\bf \Delta} \Gamma(f)-\Gamma(f, \Delta f)\,\mm}_{={\bf \Gamma}_2(f)}-\|\H_f\|^2_{\rm HS}\,\mm;
\]
\item  for any $f\in {\rm TestF}_\loc$, $$\bRic(\nabla f, \nabla f) \geq K\Gamma(f) \,\mm;$$
\item  for any $f,g, h\in {\rm TestF}_\loc$, $$\bRic(h\nabla f, \nabla g)=h \bRic(\nabla f, \nabla g).$$ 
\end{enumerate}
\end{proposition}

\subsection{Equality  in the 2-Bakry-\'Emery inequality}

In the  next  lemma, we study  the equality in the  2-Bakry-\'Emery inequality. The argument for  the proof  is standard,   we just need to pay attention to  the regularity issues  appearing in the non-smooth framework.

\begin{lemma}[Equality  in the 2-Bakry-\'Emery inequality]\label{lemma}
Let $\ms$ be a  $\rcd$ probability  space for some $K> 0$  and  let $u\in \V \cap {\rm D}(\Delta) $ be a non-constant function with $\Delta u\in \V$ and  $\int u\,\d \mm=0$.
Then the following statements are equivalent.
\begin{enumerate}
\item   $u\in {\rm TestF}_\loc$ and   ${\bf \Gamma}_2(u)=K\Gamma(u)\, \mm$;
\item      $\Gamma_2(u; \varphi)=K\int \varphi\Gamma(u)\,\d \mm$ for all non-negative   $\varphi\in L^\infty$ with $\Delta \varphi \in L^\infty$;
\item  $\int (\Delta u)^2\,\d \mm=K\int \Gamma(u)\,\d \mm$;
\item  $-\Delta u=K u$;
\item  $\int \Gamma(u)\dm=K\int u^2\dm$;
\item  $\Gamma(P_t u)= e^{-2Kt} P_t {\Gamma(u)}$ for some $t>0$.
\end{enumerate}
In particular, $P_s u$ satisfies the properties above for all $s>0$. Furthermore, $P_{s} u$ satisfies one of these properties for all $s\in [0, t]$ if and only if
\[
\int (P_{t} u)^2 \dm=e^{-2Kt} \int u^2 \dm.
\]

If $u$ attains the equality in the 2-Bakry-\'Emery inequality (6)  above, it holds
\begin{itemize}
\item  [a)] $|\nabla P_t u|= e^{-Kt} P_t |\nabla u|$ for all  $t>0$;
\item   [b)] $u$ is a non-constant affine function, this means $\H_u=0$ and $\Gamma(u)$ is a positive  constant;
\item  [c)]  $u\in {\rm TestF}_\loc$ and $\bRic(u, u)=K\Gamma(u)\dm$;
\item  [d)] the gradient flow of $ u$ induces a one-parameter semigroup of isometries of $(X, \d)$.
\end{itemize}

\end{lemma}
\begin{proof}

{\bf Part 1}:
We will prove  (1) $\Longrightarrow$ (2) $\Longrightarrow$ (3) $\Longrightarrow$  (4) $\Longrightarrow$  (5) $\Longrightarrow$ (4)  $\Longrightarrow$ (6)  $\Longrightarrow$ (2).  Statement (1) is a consequence of b) and c) which will be proved in {\bf Part 2}.
 
(1) $\Longrightarrow$ (2):  Integrating  $\varphi$ w.r.t. the measures ${\bf \Gamma}_2(u)$, $K\Gamma(u)\, \mm$ we get the answer.

(2) $\Longrightarrow$ (3): Notice that  the constant function $\varphi \equiv 1$ is admissible, and  $\Gamma_2(u, 1)= \int (\Delta u)^2\,\d \mm$.

 (3) $\Longrightarrow$ (4): Applying  Proposition \ref{def-bevar}  with $\varphi \equiv 1$ (or by Proposition \ref{prop: self}, (2)),  we can see that  $$\int (\Delta f)^2\,\d \mm\geq K\int \Gamma(f)\,\d \mm$$ for  $f \in {\rm D}(\Delta)$. Let $f=u\pm\epsilon g$ for some $g\in {\rm D}(\Delta)$ and $\epsilon \in \R$. We obtain 
\begin{equation}\label{eq1:lemma}
\int \big(\Delta (u \pm \epsilon g ) \big )^2\,\d \mm\geq K\int \Gamma(u\pm \epsilon g)\,\d \mm.
\end{equation}

Differentiating \eqref{eq1:lemma} (w.r.t.  the variable $\epsilon$), and combining with the equality in (3) we get 
\[
\pm \int \Delta u \Delta g\,\d \mm\geq \pm K\int \Gamma(u,  g)\,\d \mm.
\]
Therefore
\begin{equation}\label{eq2:lemma}
\int \Delta u \Delta g\,\d \mm= K\int \Gamma(u,  g)\,\d \mm=-K\int u \Delta g\, \d\mm.
\end{equation}
Notice that  ${\rm D}(\Delta)$  is dense in $\V$, and by Poincar\'e inequality it holds  $$\overline{\Delta \big( {\rm D}(\Delta) \big)}^{L^2}\bigoplus \Big\{u\equiv c: c\in \R,  c\neq 0 \Big \}=L^2.$$ Hence    \eqref{eq2:lemma} yields (4).

(4) $\Longrightarrow$ (5)  Multiplying $u$ on both sides of $-\Delta u=K u$ and integrating w.r.t. $\mm$, we obtain  the equality in the  Poincar\'e inequality.

(5) $\Longrightarrow$ (4):  By  Poincar\'e inequality, we have $\int \Gamma(u+g)\dm\geq K\int (u+g)^2\dm$ for all $g\in \V$ with $\int g\dm=0$.  Then similar to (3) $\Longrightarrow$ (4),    we can prove the spectral gap equality by   a standard variation argument.

(4) $\Longrightarrow$ (6):  Denote  $\phi(t):=   \int \Big (\Gamma(P_t u)-e^{-2Kt} P_t \Gamma(u) \Big)\dm$.  By (4) we have $-\Delta P_t u=K P_t u$ for any $t \geq 0$, so $\int (\Delta P_t u)^2\,\d \mm=K\int \Gamma(P_t u)\,\d \mm$. It is known that (c.f. \cite[Lemma 2.1]{AGS-B}) $\phi \in C^1$, and
\begin{eqnarray*}
\phi'(t)&=&2 \int \Big(-\big (\Delta P_t u \big)^2+Ke^{-2Kt}  \Gamma(u) \Big)\dm\\
&=& 2 \int \Big(-K \Gamma(P_t u) +Ke^{-2Kt} \Gamma(u) \Big)\dm\\
&\geq&  0.
\end{eqnarray*}
Therefore $\phi(t)\geq \phi(0)=0$. Note  that by  2-Barky-\'Emery inequality $\Gamma(P_t u) \leq e^{-2Kt} P_t \Gamma(u)$, it holds $\phi\leq 0$. So  $\phi\equiv0$ and $\Gamma(P_t u) = e^{-2Kt} P_t \Gamma(u)$ for all $t>0$  which is the thesis.

(6) $\Longrightarrow$  (2):  It is known (c.f. \cite[Lemma 2.1]{AGS-B}) that $[0, t]  \ni s  \mapsto \Phi_{t, \varphi} (s) :=\frac12 \int  e^{-2Ks}P_s \varphi \Gamma(P_{t-s} u)\,\d \mm$ is $C^1$-continuous  for any  positive $\varphi \in L^\infty$ with $\Delta \varphi \in L^\infty$, and 
\[
\Phi_{t, \varphi} '(s)=e^{-2Ks} \Big(\Gamma_2(P_{t-s} u; P_s\varphi)-K \int P_s\varphi \Gamma(P_{t-s} u) \,\d \mm \Big) \geq 0.
\]
By 2-Bakry-\'Emery inequality,  (6) holds  if and only if $\Phi_{t, \varphi} '(s)=0$ for  any $s\in [0, t]$ and  any admissible function $\varphi$, i.e.
 \begin{equation}\label{eq3:lemma}
\Gamma_2(P_{t-s} u; P_s\varphi)=K \int P_s\varphi \Gamma(P_{t-s} u) \,\d \mm, ~~~~~\forall s\in [0, t].
 \end{equation}
 Notice that $u$ attains the equality in the  2-Bakry-\'Emery inequality  for $t>0$ if and only if it holds for all $t'\in [0, t]$, thus \eqref{eq3:lemma} implies
 \begin{equation}\label{eq4:lemma}
\Gamma_2 (P_{s} u; \varphi)=K \int \varphi \Gamma(P_{s} u) \,\d \mm, \qquad  \varphi, \Delta \varphi \in L^\infty,~0\leq s \leq t
 \end{equation}
which yields (2).

\bigskip

{\bf Part 2}:
Let $u_s=P_s u$. If $u$ satisfies one of the properties (1)-(6), from the discussion in the first part we know $\Delta u=-K u$. So $\Delta u_s=P_s \Delta u=-K u_s$, and $u_s$ also satisfies these properties.

Note that $\frac{\d}{\d s} u_s=\Delta u_s$.
By Poincar\'e inequality, we get
\begin{eqnarray*}
\frac{\d}{\d s} \frac12 \int  u_s^2 \dm &=& \int  u_s \frac{\d}{\d s}  u_s \dm\\
&=&\int u_s \Delta  u_s\dm\\
&=& -\int  \Gamma( u_s) \dm\\
&\leq& -K \int u_s^2 \dm.
\end{eqnarray*}
By Gr\"onwall's lemma, we obtain
\begin{equation}\label{lemma-2.31}
\int u_s^2 \dm \leq e^{-2Ks} \int u^2 \dm.
\end{equation}

Therefore, \eqref{lemma-2.31} is an equality  for some $t>0$ if and only if   $u_s $ attains the equality in the Poincar\'e inequality (5).

\bigskip
{\bf Part 3}:
 Furthermore,  by 1-Bakry-\'Emery inequality    and Cauchy-Schwarz inequality,  we have
 \[
 |\nabla P_t u|\leq  e^{-Kt} P_t |\nabla u| \leq  e^{-Kt} \sqrt {P_t\Gamma(u) }.
 \]
 So  if $u$ attains the equality in the 2-Bakry-\'Emery inequality  (6), it holds $ |\nabla P_t u|=  e^{-Kt} P_t |\nabla u| $.

 In addition,  integrating the non-smooth Bochner inequality in  Proposition \ref{prop:measurebochner} we obtain
 \[
 \int (\Delta u)^2\dm \geq K\int \Gamma(u)\dm+\int \| \H_u \|^2_{\rm HS}\dm.
 \]
 Thus the validity of (3) yields $\H_u=0$. In particular,  for any $v\in \V$, it holds $$\Gamma\big(\Gamma(u), v \big)=2 \H_u(\nabla u, \nabla v)=0,$$
 so $\Gamma(\Gamma(u))=0$ and
  $\Gamma(u)= |\nabla u|^2\equiv c$ for some constant $c\geq 0$. In particular,  $u\in {\rm TestF}_\loc$.   If $c=0$,  $f$ is  constant. 
If $c\neq 0$, by \cite[Theorem 1.2]{Han-CVPDE18} we know that the regular Lagrangian flow $(F_r)_{r\in \R^+}$  associated with $ \nabla u$   induces  a family of  isometries, i.e. $\d(F_r(x), F_r(y))=\d(x, y)$ for any $x, y \in X$ and $r>0$.
 
 Furthermore, by definition of $\bRic$ (c.f.  Proposition \ref{prop:measurebochner}) and  statement (2) proved in {\bf Part 1},  for any $\varphi \in L^\infty \cap {\rm D}(\Delta)$ with $\Delta \varphi \in L^\infty$ we have
 \[
 \Gamma_2(u; \varphi)=\int \varphi \| \H_u\|^2_{\rm HS}\dm+\int \varphi \,\d \bRic(u, u)=K\int \varphi \Gamma(u)\dm.
 \]   
Combining with $\H_u=0$ we obtain $${\bf \Gamma}_2(u)=\bRic(u, u)=K\Gamma(u)\,\mm$$  and we  complete the proof.

\end{proof}

The following proposition  plays a key role in studying  $\Phi$-entropy inequalities in \S \ref{sect:phi}.

\begin{proposition}\label{theorem:phi}
Let  $\ms$ be a  metric measure space satisfying  $\rcd$ condition for some $K>0$. Let $\Phi$ be a $C^2$-continuous   convex function on an interval $I\subset \R$ such that  $\frac{1}{\Phi''}$
is concave and strictly positive.  Then for all $t>0$, we have
\begin{equation}\label{eq0:lemma:phi}
{\Phi''(P_t u)}{ {\Gamma(P_t u )}}   \leq  e^{-2Kt} P_t \big({\Phi''}(u){ \Gamma(u)} \big)
\end{equation}
for any  $I$-valued function $u\in \V$.
In particular, the function $t\mapsto e^{2Kt} \int {\Phi''(P_t u)}{ {\Gamma(P_t u )}} \dm$ is non-increasing.

Furthermore,  the equality holds in \eqref{eq0:lemma:phi}  if and only if the following  properties are satisfied.

\begin{enumerate}
\item $(\Phi'')^{-1}$ is affine  on  the image of $u$ which is defined as $\supp u_\sharp \mm$ (by Lemma \ref{lemma:image} below we know $\supp u_\sharp \mm$ is a closed interval or a point).
\item    For any  $s\in [0, t]$, there is a constant $c=c(s)>0$ with $c(s)=e^{-2Ks}c(0)$,  such that  $$\sqrt{\Gamma(P_s u)}= e^{-Ks} P_s {\sqrt{\Gamma(u)}}~~\text{ and~}~~ \Gamma\big ( \Phi'(P_s u) \big)= c.$$
\end{enumerate}

\end{proposition}
\begin{proof}
Denote $P_t u$ by $u_t$. We have the following 1-Bakry-\'Emery inequality, 
\begin{equation}\label{eq1:lemma:phi}
\sqrt{\Gamma(u_t)} \leq e^{-Kt} P_t {\sqrt{\Gamma(u)}},\qquad\forall~t\geq 0,~~\forall u\in \V.
\end{equation}

By concavity of  $\frac{1}{\Phi''}$ and Jensen's inequality,  we have
\begin{equation}\label{eq1.5:lemma:phi}
\Phi''(u_t) \leq \Big (P_t \big( 1/{\Phi''(u)}\big)\Big)^{-1}.
\end{equation}
Combining with  \eqref{eq1:lemma:phi} we  get the following inequality
\begin{equation}\label{eq2:lemma:phi}
\Phi''(u_t)\Gamma(u_t) \leq  e^{-2Kt} \Big (P_t {\sqrt{\Gamma(u)}} \Big)^2 \Big (P_t \big( 1/{\Phi''}(u)\big)\Big)^{-1}.
\end{equation}

By Cauchy-Schwarz inequality we know
\begin{equation}\label{eq3:lemma:phi}
\Big (P_t {\sqrt{\Gamma(u)}} \Big)^2  \leq \Big(P_t \big( {\Phi''}(u){ \Gamma(u)}\big) \Big) \Big (P_t \big( 1/{\Phi''}(u)\big)\Big).
\end{equation}

Combining  \eqref{eq2:lemma:phi} and \eqref{eq3:lemma:phi}, we obtain
\begin{equation}\label{eq4:lemma:phi}
\Phi''(u_t)\Gamma(u_t)  \leq  e^{-2Kt}  P_t \big( {\Phi''}(u){ \Gamma(u)}\big)
\end{equation}
which is \eqref{eq0:lemma:phi}. Integrating \eqref{eq4:lemma:phi} w.r.t. $\mm$,  we obtain 
\begin{equation}\label{eq4.1:lemma:phi}
e^{2Kt} \int {\Phi''(u_t)}{ {\Gamma(u_t )}} \dm  \leq  \int {\Phi''}(u){ \Gamma(u)} \dm.
\end{equation}
By semigroup property, we can see that $e^{2Kt} \int {\Phi''(u_t)}{ {\Gamma(u_t)}} \dm$ is non-increasing in $t$.
\bigskip

Therefore, the equality in \eqref{eq0:lemma:phi} holds for some $t_0$ if and only if
\begin{equation}\label{eq4.11:lemma:phi}
e^{2Kt_0} \int {\Phi''(u_{t_0})}{ {\Gamma(u_{t_0} )}} \dm  =  \int {\Phi''}(u){ \Gamma(u)} \dm.
\end{equation} Furthermore, it holds
\begin{equation}\label{eq4.12:lemma:phi}
e^{2Kt} \int {\Phi''(u_{t})}{ {\Gamma(u_{t} )}} \dm  =  \int {\Phi''}(u){ \Gamma(u)} \dm.
\end{equation}
 for any  $t\leq t_0$. Hence the equality in \eqref{eq0:lemma:phi} holds  for some $t_0>0$  if and only if  the equalities in \eqref{eq1:lemma:phi} \eqref{eq1.5:lemma:phi} and \eqref{eq3:lemma:phi} hold for all $0\leq t\leq t_0$. The equality in \eqref{eq1.5:lemma:phi} holds iff $(\Phi'')^{-1}$ is affine on the image of $u$, and the validity of the equality in \eqref{eq3:lemma:phi} if and only if

\begin{equation}\label{eq5:lemma:phi}
\Phi''(u_t)\Gamma\big ( u_t \big)= \frac c{\Phi''(u_t)}
\end{equation}
for some constant $c=c(t)>0$.     Moreover,   for any $t\leq t_0$ we have
\[
\sqrt{c(t)}\mathop{=}^{\eqref{eq5:lemma:phi}}\Phi''(u_t)\sqrt{\Gamma\big ( u_t \big)}\mathop{=}^{\eqref{eq1:lemma:phi} \eqref{eq1.5:lemma:phi} }  e^{-Kt} \frac{ P_t \big(  \sqrt{\Gamma( u)} \big)}{P_t \big(1/\Phi''(u) \big)} \mathop{=}^{\eqref{eq5:lemma:phi}} e^{-Kt}  \sqrt{c(0)}
\]
which is the thesis.

\end{proof}

\begin{lemma}\label{lemma:image} 
Let $\ms$ be an $\rcd$ metric measure space and $u\in \V$. Then  the image of $u$,  defined as $\supp u_\sharp \mm$,  is a closed interval in $\R$ or a point in which case $u$ is constant.
\end{lemma}
\begin{proof}
Denote ${\rm ess} \sup u=b\in \R \cup \{+\infty\}$ and ${\rm ess} \inf u=a\in \R \cup \{-\infty\}$. We will show that $\supp u_\sharp \mm=[a, b]$.  

If $a=b$, $u$ is constant, the assertion is obvious. Otherwise, $a<b$.
For any $c\in (a, b)$ and $\epsilon>0$ small enough such that $(c-\epsilon, c+\epsilon) \subset (a+\epsilon, b-\epsilon)$. Pick bounded measurable sets $A, B \subset X$ with  positive $\mm$-volume  such that $A \subset u^{-1}\big((a, a+\epsilon) \big)$ and 
$B\subset u^{-1}\big((b-\epsilon, b) \big)$.   By \cite{GRS-O} there is  a unique  $L^2$-Wasserstein geodesic  $(\mu_t)$ from $\mu_0:=\frac{\chi_A}{\mm(A)}\mm$ to $\mu_1:=\frac{\chi_B}{\mm(B)}\mm$. 
 There is  $\Pi \in {\mathcal P}_2(\geo(X, \d))$ such that $(e_t)_\sharp \Pi=\mu_t$ (c.f. \cite[Theorem 2.10]{AG-U}).  By  \cite[Lemma 3.1]{Rajala12}  we know $\frac{\d \mu_t}{\d \mm}$ is uniformly bounded, so $\Pi$ is a test plan (in the sense of  \cite[Definition 5.1]{AGS-C}).  By an equivalent characterization of Sobolev functions using test plans (c.f. \cite[\S 5, Proposition 5.7 and \S 6]{AGS-C}), we know
$u\circ \gamma \in W^{1,2}([0, 1])$  for $\Pi$-a.e. $\gamma$. Hence  for $\Pi$-a.e. $\gamma$, the map $t \mapsto u\circ \gamma(t)$ has an absolutely continuous representative.   So there is a set  $I_\gamma \subset [0, 1]$  with positive $\mathcal L^1$-measure  such that $u\circ \gamma(I_\gamma) \subset (c-\epsilon, c+\epsilon)$. By Fubini's theorem again, there is $t_c \in (0,1)$ and  $\Gamma_c \subset \supp \Pi$ with positive measure,  such that $u\circ \gamma(t_c) \in (c-\epsilon, c+\epsilon)$ for all $\gamma \in \Gamma_c$. Therefore $$\mu_{t_c}\Big (\big \{\gamma(t_c): \gamma\in \Gamma_c\big \} \Big)=(e_{t_c})_\sharp \Pi\restr{\Gamma_c}(X)>0.$$

From the definition of $\cd$ condition (see \eqref{eq1.5-intro}) we know $\mu_{t_c} \ll \mm$, so 
$$u_\sharp \mm\big((c-\epsilon, c+\epsilon)\big)=\mm\Big(u^{-1}\big((c-\epsilon, c+\epsilon)\big) \Big)\geq \mm \Big (\big \{\gamma(t_c): \gamma\in \Gamma_c \big \} \Big)>0.$$
Hence $c\in  \supp u_\sharp \mm$. Since  the choice of $c$ is arbitrary and $\supp u_\sharp \mm$ is closed, we know $\supp u_\sharp \mm=[a, b]$. 
\end{proof}

\begin{corollary}\label{coro:phi}
Under the same assumption as Proposition \ref{theorem:phi}, if there exists a non-constant  $u\in \V$ attaining the equality in \eqref{eq0:lemma:phi}   for all $t>0$, then up to additive and multiplicative constants,  and  affine  coordinate transforms,  $\Phi(x)=x\ln x$ or $\Phi(x)=x^2$.
 In any of these cases,  the function $P_t u$ attains the equality in the 1-Bakry-\'Emery inequality and the function $ \Phi'(P_t u)$ attains  the equality in the Poincar\'e inequality.  In particular,   $ \Phi'(P_t u) -\int  \Phi'(P_t u)\dm$ satisfies the properties (1)-(6)  in Lemma \ref{lemma} for all $t>0$.
\end{corollary}
\begin{proof}
By Proposition \ref{theorem:phi} and Lemma \ref{lemma:image} we know $(\Phi'')^{-1}$ is linear on an interval $I$.  So for  $x\in I$,   $\Phi''(x)=\frac1{c_1 x+c_2}$ for some constants $c_1, c_2$. 
 If $c_1=0$, $\Phi=x^2$ up to an additive constant and an  affine  coordinate transformation. If $c_1\neq 0$, up to an affine coordinate transform, $\Phi$ can be written as $x\ln x+c_3x+c_4$.  In the latter case, we can write $\Phi$  as $\Phi(x)=\frac1{e^{c_3}} \big((e^{c_3}x) \ln (e^{c_3}x )\big)+c_4$, which is the thesis.

Furthermore, by  Proposition \ref{theorem:phi} we know  $\Gamma (u_s)= c(s) /\big(\Phi''(u_s)\big)^2$ for all $s>0$, and $ c(s)=e^{-2Ks} c(0)$.  Thus for any $t>0$,  we have
\begin{eqnarray*}
&&\int \big(\Phi'(u_t)\big)^2 \dm- \Big(\Phi'\big(\int u \dm\big)\Big)^2 \\
&=& \int_{+\infty}^t \frac{\d }{\d s} \int \big(\Phi'(u_s)\big)^2 \dm\,\d s\\
\text{By \cite[Theorem 4.16]{AGS-C}}&=&\int^{+\infty}_t \int 2\Big(\big(\Phi''(u_s)\big)^2 +\Phi'(u_s)\Phi^{(3)}(u_s)\Big)\Gamma(u_s) \dm\, \d s  \\
&=& \int^{+\infty}_t  c(s) \int 2\Big(1 +\frac{\Phi'\Phi^{(3)}}{(\Phi'')^2}(u_s)\Big)\dm\,\d s \\
&=& \int^{+\infty}_t   2e^{-2K(s-t)} c(t) \int \Big(1 +\frac{\Phi'\Phi^{(3)}}{(\Phi'')^2}(u_s)\Big)\dm\,\d s.
\end{eqnarray*}

Similarly,
\begin{eqnarray*}
&&\Big(\int \Phi'(u_t) \dm\Big)^2- \Big(\Phi'\big(\int u \dm\big)\Big)^2 \\
&=& \int_{+\infty}^t \frac{\d }{\d s} \Big(\int \Phi'(u_s) \dm\Big)^2\,\d s\\
&=&\int^{+\infty}_t 2\Big (\int \Phi'(u_s) \dm\Big) \int \Phi^{(3)}(u_s)\Gamma(u_s) \dm\, \d s  \\
&=& \int^{+\infty}_t  2c(s)\Big (\int \Phi'(u_s) \dm\Big) \int \frac{\Phi^{(3)}}{(\Phi'')^2}(u_s)\dm\,\d s \\
&=& \int^{+\infty}_t   2e^{-2K(s-t)} c(t) \Big (\int \Phi'(u_s) \dm\Big) \int   \frac{\Phi^{(3)}}{(\Phi'')^2}(u_s)\dm\,\d s.
\end{eqnarray*}

Since $(\Phi'')^{-1}$ is linear,  we can see  that $\eta:=\frac{\Phi^{(3)}}{(\Phi'')^2}=-\Big (\frac{1}{\Phi''} \Big )' $ is  constant, so 
\begin{eqnarray*}
&&\int \big(\Phi'(u_t)\big)^2 \dm-\Big(\int \Phi'(u_t) \dm\Big)^2\\
&=& \int^{+\infty}_t   2e^{-2K(s-t)} c(t) \Big( \int \big(1 +\eta \Phi'(u_s) \big) \dm- \eta \int  \Phi'(u_s)\dm \Big )\,\d s\\
&=& \int^{+\infty}_t   2e^{-2K(s-t)} c(t)  \,\d s\\
&=& \frac1K c(t)\\
 &=& \frac1{K} \int \Gamma\big(\Phi'(u_t)\big) \dm.
\end{eqnarray*}
This means that  $\Phi'(u_t)$ attains  the equality in the  Poincar\'e inequality.

\end{proof}

\subsection{One-dimensional cases}
In this part, we will prove  the   rigidity of the 2-Bakry-\'Emery inequality in 1-dimensional cases.  This result is  a simple application of Lemma \ref{lemma},  and it will be used in  the study of higher-dimensional spaces.

\begin{proposition}\label{prop:1dimrigidity}
Let $h$ be a  $\cd$  probability density supported on a closed set   $I \subset \R$, this means,  $ h\mathcal L^1$ is a probability measure such that $(I, |\cdot|, h\mathcal L^1)$ is a $\cd$ space.  If there is a  non-constant function $f$ satisfying one of the properties (1)-(6) in Lemma \ref{lemma}, then $I=\R$ and  $h(t)=\phi_K(t)=\sqrt{\frac{K}{2\pi}} \exp(-\frac{Kt^2}2)$ up to a translation. Furthermore,  there is a constant $C=|f'|>0$ such that 
\[
P_t f(x)=C e^{Kt} x,\qquad\forall~~t\geq 0.
\]

\end{proposition}
\begin{proof}
Since $h$ is a  $\cd$ density,  it is known (c.f.  \cite{AGS-G, Han-CDonRM}) that $-\ln h$ is $K$-convex and $\supp h$ is a closed interval $I:=[a, b]$ with $a\in \R \cup \{-\infty\}$ and $b\in \R\cup \{+\infty\}$.  In particular,  $h$ is  locally Lipschitz.   By Rademacher's  theorem,  $h'(x)$ exists for $\mathcal L^1$-a.e. $x\in I$. Furthermore, $(\ln h)'$ is a BV function and  $-(\ln h)'' \geq K$  in weak sense, i.e.
\begin{equation}\label{eq1:prop:1dim}
\int \varphi' (\ln h)'\, \d \mathcal L^1 \geq K\int \varphi \, \d \mathcal L^1
\end{equation} for all $\varphi \in C^1_c$ with $\varphi \geq 0$ and $\varphi'(a)=\varphi'(b)=0$.
\bigskip

Consider the $\Gamma_2$-calculus on the  metric measure space $(I, |\cdot|, h\mathcal L^1)$.  For $f\in {\rm D}(\Delta_I)$, by Proposition \ref{prop: self} and the fact that $\Gamma$ operator on  $(I, |\cdot|, h\mathcal L^1)$ coincides with the usual derivative, we know $f'\in W^{1,2}(I)$.  So it is absolutely continuous, and $f''(x)$  exists at almost every $x\in I$. By assumption and Lemma \ref{lemma}, we know $\H_f=0$.  By \eqref{eq:hessian}, we know  $\H_f= f''=0$ and $f'$ is  constant. By integration by part formula and Newton-Leibniz formula, for any $\varphi\in C^1_c$ we have
\begin{eqnarray*}
\int_I \varphi \Delta_I f\, h\d\mathcal L^1&=&-\int \varphi' f' \,h\d\mathcal L^1\\
&=& \int_I \varphi \big(f''-(\ln h)' f'\big)\, h\d\mathcal L^1+\varphi f' h\big( \delta_a-\delta_b\big).
\end{eqnarray*} By definition $\Delta_I f \in L^2$, so we have  $f'\restr{\{a, b\}\setminus\{\pm\infty\}}=0$, and 
\begin{equation}\label{1dim-1}
\Delta_h f=f''-(\ln h)' f'=-(\ln h)' f'.
\end{equation}
 Since $f$ is not  constant, there must be 
$\{a, b\}=\{\pm\infty\}$ and $I=\R$.

By \eqref{1dim-1} and \eqref{eq1:prop:1dim}, for any $\varphi \in C^1_c$, we have 
\begin{eqnarray*}
\int \big(\Delta_h f \big)^2\varphi h \,\d \mathcal L^1&=&\int \big((\ln h)' f'\big)^2 \varphi h\,\d \mathcal L^1\\ &=&(f')^2\int (\ln h)' \varphi h'\,\d \mathcal L^1\\ &\geq& K\int (f')^2 \varphi h\,\d \mathcal L^1-(f')^2\int (\ln h)' \varphi'h)\,\d.
\end{eqnarray*}
Letting $\varphi \to 1$ we get
\begin{equation}\label{eq2:prop:1dim}
\int \big(\Delta_h f \big)^2 h \,\d \mathcal L^1\geq K\int (f')^2 h\,\d \mathcal L^1.
\end{equation}
By assumption,  the equality holds in \eqref{eq2:prop:1dim}. Hence 
 there must be  $(\ln h)''=K$ in usual sense.  Up to a translation,  $h(x)= \sqrt{\frac{K}{2\pi}} \exp(-\frac{Kx^2}2)=\phi_K(x)$ for  $x \in \supp h=\R$. 

\bigskip
Furthermore,
by Lemma  \ref{lemma} we  have  $(P_t f)''=0$,  and $(P_t f)'$ is   constant  for any $t\geq 0$.   So there exist smooth  functions   $a=a(t), b=b(t) \in \R$ such that 
$$P_t f(x)=a(t)x+b(t).
$$

Notice  that  $\ddt P_t f =(P_t f)''-(\ln h)' (P_t f)'$,  we have
\[
\ddt a(t)x+\ddt b(t)=Kx a(t).
\]
Hence $a(t)=C e^{Kt}$ with $C=|f'|>0$, and $b\equiv 0$.

\end{proof}

\section{Rigidity of the 1-Bakry-\'Emery inequality}\label{sec:be}

\subsection{Equality in the 1-Barky-\'Emery inequality}\label{section:localization}

 In this part, we will prove  one of the most important results  in this paper, concerning  the equality in the 1-Bakry-\'Emery inequality. Several intermediate steps,   which corresponds  to  the results in  \cite[\S 2]{ABS19}  of Ambrosio-Bru\'e-Semola,   will be proved in    separate lemmas  before the main Theorem \ref{thm:localization}.    We remark that some arguments used   in  \cite{ABS19}  concerning  ${\rm RCD}(0, N)$ spaces are not  available now. For example,  there is no two-sides heat kernel estimate or uniform volume doubling property for general  $\rcd$  spaces.   Fortunately, we can  overcome these difficulties  by  making full use of the heat flow and the functional analysis tools developed by Gigli in  \cite{G-N}.
 
 \begin{lemma}\label{exp4}
 Let $\ms$ be an $\rcd$ probability space with $K \in \R$. Assume there exists a non-constant function $ f\in \V$ satisfying
$$|\nabla P_{t_0} f| =e^{-Kt_0}P_{t_0}|\nabla f| \quad\quad~~~\text{for some}~~ t_0>0.$$
For any  $s\in (0, t_0)$, denote $$A_s:= \Big\{|\nabla P_s f|=0 \Big\}.$$  Then it holds
$$\mm(A_s)=0.$$
In particular,
\[\mm\Big(\big\{ P_s f=c\big\} \Big)=0,\qquad \forall~c\in \R.
\]
\end{lemma}
\begin{proof}
Assume by contradiction that   $\mm (A_s )>0$ for some  $s>0$.    Since $f$ is non-constant,  we know $\mm(A_s)\in (0, 1)$.

  Recall that $f$ attains the equality in the 1-Barky-\'Emery inequality,  we have 
\[
P_s|\nabla f|=e^{Ks} |\nabla P_s f|=0,\qquad\text{on}~ A_s.
\]
Thus
\[
0=\int_{A_s} P_s |\nabla f|\dm=\int P_s(\nchi_{A_s}) |\nabla f|\dm.
\]

Denote $A_0^c:= \Big\{|\nabla  f| >0 \Big\}$. We can see that 
\begin{equation}\label{1-exp3}
\int_{A_0^c}  P_{ s} ( \nchi_{A_s}) \dm=0,
\end{equation}
i.e.  $P_{ s} ( \nchi_{A_s})=0$ on $A_0^c$.
Note that  $P_s(\nchi_{A_s})$ is Lipschitz continuous, and by dimension-free Harnack inequality on $\rcd$ spaces proved by H.-Q. Li in \cite[Theorem 3.1]{Li-Harnack},  it holds
\[
\big ((P_s \nchi_{A_s} )(y) \big) ^2 \leq (P_s\nchi_{A_s} )(x)\exp{\Big\{\frac{K\d^2(x, y)}{e^{2Ks}-1}\Big\}}.
\]
So  $P_{ s} ( \nchi_{A_s})(x)>0$ at every point $x \in X$.  Thus  $\mm(A_0^c)=0$ and $\mm(A_0)=1$, which contradicts to the assumption that $f$ is non-constant.

Finally, by locality of the weak gradient (c.f. \cite[Proposition 5.16]{AGS-C}),  it holds $|\nabla P_s f|=0$ $\mm$-a.e. on $\{P_s f=c\} $. So $\mm\big (\{P_s f=c\} \big)\leq \mm(A_s)=0$.

\end{proof}

\begin{lemma}\label{lemma:1be}
Under the same assumption as Lemma \ref{exp4}.
Denote $b_s:=\frac{\nabla P_s f}{e^{-Ks}|\nabla P_s f|}$.  Then  for any $g\in \V$ and $s, t \in \R^+$ with $s+t<t_0$, it holds $$\la b_{t+s}, \nabla P_t g \ra=P_t \la b_s, \nabla g \ra.$$ 
\end{lemma}

\begin{proof}
 By 1-Bakry-\'Emery inequality and the assumption,  for any  $s, t, r \in (0, t_0)$ with $s+t+r=t_0$,  we can see that
\begin{eqnarray*}
0&\geq&
 e^{-Kr} P_r \Big(|\nabla P_{t+s} f| -e^{-Kt}P_t|\nabla P_s f|\Big)\\
&=& \Big(e^{-Kr} P_r  |\nabla P_{t+s} f| -\underbrace{e^{-K(t+s+r)}P_{t+s+r}|\nabla  f|}_{e^{-Kt_0}P_{t_0}|\nabla f|}\Big)\\&& + \Big(e^{-K(t+s+r)}P_{t+s+r}|\nabla  f| -e^{-K(t+r)}P_{t+r}|\nabla P_s f|\Big)\\
&=&\Big(e^{-Kr} P_r  |\nabla P_{t+s} f| -\underbrace{|\nabla P_{t+s+r} f|}_{|\nabla P_{t_0} f|}\Big) + \Big(e^{-K(t+s+r)}P_{t+s+r}|\nabla  f| -e^{-K(t+r)}P_{t+r}|\nabla P_s f|\Big)\\
&\geq& 0.
\end{eqnarray*}
 Thus 
 \begin{equation}\label{1be-1}
 |\nabla P_{t+s} f| =e^{-Kt}P_t|\nabla P_s f|
 \end{equation}
 for any $s, t\in \R^+$ with $s+t<t_0$ (c.f. \cite[Lemma 2.4, 2.7]{ABS19}). 

 Fix $t>0$ and consider the Euler equation associated with the functional 
\[
\Psi (h):= \int \big( e^{-Kt} P_t|\nabla h|-|\nabla P_t h| \big) \varphi\dm,\qquad h\in \V, \varphi \in \Lip_{bs}(X, \d).
\]

From Lemma \ref{exp4} we know $\frac{\nabla P_s f}{|\nabla P_s f|}$ is well-defined and  $\left |\frac{\nabla P_s f}{|\nabla P_s f|} \right |=1$ $\mm$-a.e.. 
Using a standard variation argument  (c.f.  \cite[proof of Proposition 2.6]{ABS19}),  for any $g\in \V$ and  $s>0$ with $s+t<t_0$,  we get
\begin{eqnarray*}
0&=& \frac {\d}{\d \epsilon} \restr{\epsilon=0} \Psi(P_s f+\epsilon g)\\
&=& \int \Big (e^{-Kt} P_t \big (\frac{\la \nabla P_s f,  \nabla g \ra}{|\nabla P_s f|} \big) -\frac{\la \nabla P_{t+s} f, \nabla P_t g\ra}{|\nabla P_{t+s} f|} \Big )\varphi \dm\\
&=& e^{-K(t+s)} \int \Big ( P_t \la b_s, \nabla g  \ra -\la b_{t+s}, \nabla P_t g \ra \Big) \varphi \dm.
\end{eqnarray*}
 Then the conclusion follows from the arbitrariness of $\varphi$.

\end{proof} 

\begin{lemma}\label{lemma:div0}
Let $\ms$ be an $\rcd$ probability  space.   Assume  there is a non-constant function  $f\in \V$ satisfying
$$|\nabla P_{t_0} f| =e^{-Kt_0}P_{t_0}|\nabla f| \qquad \text{for}~~ t_0>0, $$
and denote $b_s:=\frac{\nabla P_s f}{e^{-Ks}|\nabla P_s f|}$.

Then  $ b_s \in {\rm D}(\div)$ for any $s\in (0, t_0)$. Furthermore,   for any $s,t>0$ with $s+t<t_0$,
\begin{equation}\label{div-1}
P_t \div(b_{t+s})=\div( b_s).
\end{equation}
In particular, $\div( b_s) \in {\rm D} (\Delta)$ and $\Delta\div( b_s) \in \V$.
\end{lemma}
\begin{proof}
For any $g\in \V$, we have
\begin{eqnarray*}
\left |\int \la b_{s}, \nabla g \ra \dm\right| &=& \left | \int P_t \la b_s, \nabla g\ra \dm \right |\\
\text{By  Lemma \ref{lemma:1be} }&=&  \left | \int \la b_{t+s}, \nabla P_t g\ra \dm \right |\\
&\leq&  \int  |b_{t+s}||\nabla P_t g| \dm\\
\text{By $|b_r|=e^{Kr}$ and Cauchy-Schwartz inequality} &\leq & e^{(t+s)K} \sqrt{\E( P_t g)}.
\end{eqnarray*}
Note that   it holds a standard estimate (c.f. Lemma \ref{lemma:reg})  $\E(P_t g) \leq \frac 1{2t} \| g\|^2_{L^2}$. Hence by Riesz representation theorem,  $ b_s \in {\rm D}(\div)$.

At last, the identity \eqref{div-1} follows immediately from  Lemma \ref{lemma:1be}.
\end{proof}

\begin{proposition}\label{lemma:1be2}
Keep the same assumption and notations as in Lemma \ref{lemma:div0}.  It holds
\[
\int \big( \div(b_s)\big)^2 \dm =e^{2Ks}\int \big( \div(b_0)\big)^2 \dm
\]
for all $s\in [0,  t_0]$.
\end{proposition}
\begin{proof}
{\bf Step 1}:

Given $g\in \V$.  Consider the following function $t\mapsto \psi(t, g)$ defined on $\R^+$ 
\[
\psi(t, g):= \int  e^{Kt} |\nabla P_t g|  \dm.
\]
From 1-Bakry-\'Emery inequality  we know $\psi$ is non-increasing in $t$ and it is differentiable almost everywhere.  Similar to the computation in Lemma \ref{lemma:1be},  we can see that
\[
\ddt \psi(t, g)= \int  K e^{Kt}   |\nabla P_t g| +\la b^g_t, \nabla \Delta P_t g \ra  \dm\leq 0
\]
where $b^g_t:=e^{Kt}\frac{\nabla P_t g}{|\nabla P_t g|} \in L^2(TX)$. Note also that $b_t^f=b_t$.

Fix $s\in (0, t_0)$. By assumption,  the function $t \mapsto \psi(t, P_s f)$ is  constant on $[0, t_0-s]$. So $\ddt \psi(t, P_s f)=0$ for  $t\in [0, t_0-s]$,  this means  $$\ddt \psi(t, P_s f)= \int Ke^{Kt}| \nabla  P_{t+s}  f| \dm+\int \la b^{P_s f}_{t}, \nabla \Delta P_{t+s} f\ra \dm=0\qquad \forall t\in [0, t_0-s].$$

Fix $t$ and consider the following  functional 
\[
\V \ni g \mapsto \ddt \psi(t, g)= \int K e^{Kt}  |\nabla P_t g| +\la b^g_t, \nabla \Delta P_t g \ra  \dm \leq 0
\]
which attains its maximum at $g=P_s f$.

Thus for any  $\epsilon \in \R$, 
\begin{eqnarray*}
0&\geq&  \ddt \psi(t,  P_s f+\epsilon g)-\ddt \psi(t, P_s f)\\
&=& \underbrace{ \int K e^{Kt} \Big( |\nabla P_t (P_s f+\epsilon g)| -  |\nabla P_{t+s} f|\Big)\dm}_{I} \\
&& + \underbrace{\int  \Big(\la b^{P_s f}_t, \nabla \Delta P_t (P_s f+\epsilon g) \ra -\la b_{t}^{P_s f}, \nabla \Delta P_{t+s} f \ra \Big) \dm}_{II}\\
 &&+ \underbrace{\int  \Big( \la b^{P_s f+\epsilon g}_t, \nabla \Delta P_t (P_s f+\epsilon g) \ra -\la b^{P_s f}_t, \nabla \Delta P_t (P_s f+\epsilon g) \ra \Big)  \dm}_{III}.
\end{eqnarray*}

Define ${\rm F}_t \subset \V$ by
\begin{equation}\label{1be-2.1}
{\rm F}_t:=\Big\{g: g\in \V \cap L^\infty(X, \mm), \frac{|\nabla P_t g|}{|\nabla P_{t+s}  f|} \in L^\infty(X, \mm) \Big\}.
\end{equation}
By  Lemma \ref{exp1}, 
\[
{\rm F}_0\subset {\rm F}_r \subset {\rm F}_t,\qquad \forall~ 0\leq r\leq t,
\]
and ${\rm F}_0$ is an algebra.

For any $g \in {\rm F}_t$ and $\epsilon$ small enough,  we can write $I, II, III$ in the following ways
\begin{eqnarray*}
I &=& K e^{Kt}\int \int^\epsilon_0 \frac{\la \nabla P_t (P_s f+\tau g), \nabla P_t g\ra }{|\nabla P_t (P_s f+\tau g) |}\,\d \tau\dm,
\end{eqnarray*}
\begin{eqnarray*}
II &=&\epsilon \int  \la b^{P_s f}_t, \nabla \Delta P_t  g \ra\dm,
\end{eqnarray*}
and
\begin{eqnarray*}
III &=&{e^{Kt}\int \la\frac{|\nabla P_{t+s} f| \nabla P_t (P_s f+\epsilon g)-|\nabla P_t (P_s f+\epsilon g)| \nabla P_{t+s} f}{|\nabla P_{t+s} f| |\nabla P_t (P_s f+\epsilon g) |}, \nabla \Delta P_t (P_s f+\epsilon g)  \ra \dm}\\
&=& e^{Kt}\int \la\frac{|\nabla P_{t+s} f| \nabla P_t (P_s f+\epsilon g)-|\nabla P_{t+s} f| \nabla P_{t+s} f}{|\nabla P_{t+s} f| |\nabla P_t (P_s f+\epsilon g) |}, \nabla \Delta P_t (P_s f+\epsilon g)  \ra \dm\\
&& +e^{Kt}\int \la\frac{| \nabla P_{t+s} f|\nabla P_{t+s} f - |\nabla P_t (P_s f+\epsilon g) | \nabla P_{t+s} f}{|\nabla P_{t+s} f| |\nabla P_t (P_s f+\epsilon g) |}, \nabla \Delta P_t (P_s f+\epsilon g)  \ra \dm\\
&=&\epsilon e^{Kt}\int \la\frac{ \nabla P_t  g}{|\nabla P_t (P_s f+\epsilon g) |}, \nabla \Delta P_t (P_s f+\epsilon g)  \ra \dm\\
&& +e^{Kt}\int \Big(\int_\epsilon^0 \frac{\la \nabla P_t (P_s f+\tau g), \nabla P_t g\ra }{|\nabla P_t (P_s f+\tau g) | |\nabla P_t (P_s f+\epsilon g) |}\,\d \tau \Big) \la \frac{ \nabla P_{t+s}  f}{|\nabla P_{t+s} f|}, \nabla \Delta P_t (P_s f+\epsilon g)  \ra \dm.
\end{eqnarray*}

Thus for any $g\in {\rm F}_t$, there is $\epsilon_0>0$ small enough such that the function
$\epsilon \to \ddt \psi(t, P_s f+\epsilon g)=I+II+III$ is absolutely continuous and hence differentiable on $[0, \epsilon_0]$.  

Similar to the proof of Lemma \ref{lemma:1be}, by a variational argument we get
\begin{eqnarray*}
0&=& \frac {\d}{\d \epsilon} \restr{\epsilon=0} \ddt \psi(t, P_s f+\epsilon g)\\
&=& \underbrace{\int \Big (K \la b^{P_s f}_t, \nabla P_t g \ra+\la b^{P_s f}_t, \nabla \Delta P_t g \ra\Big) \dm}_{V_t^1(\nabla P_t g)}\\
&& +\underbrace{e^{Kt}\int \Big( \frac{1}{|\nabla P_{t+s} f|}\la \nabla P_t g, \nabla \Delta P_{t+s} f\ra -\frac{1}{|\nabla P_{t+s} f|^3}\la \nabla P_{t+s} f, \nabla P_t g \ra \la \nabla P_{t+s} f, \nabla \Delta P_{t+s} f\ra\Big) \dm}_{V_t^2(\nabla P_t g)}.
\end{eqnarray*}

\bigskip

{\bf Step 2}:

Define
\[
\D_t:= {\rm Span} \left(\Big\{\nabla  g: g\in \V, \frac{|\nabla  g|}{|\nabla P_{t+s}  f|} \in L^\infty(X, \mm)\Big\} \right).
\]
where ${\rm Span}(S)$ means the sub-module of $L^2(TX)$ consisting of all finite $L^\infty$-linear combinations of the elements in $S$.
By definition of ${\rm F}_t$, we can see that
\begin{equation}\label{1be2-3}
 \Big\{ \nabla P_t g: g\in {\rm F}_t \Big\} \subset \D_t.
\end{equation}
Furthermore,  by linearity $V_t^1, V_t^2$ can be uniquely defined on $\D_t$  by:
\begin{eqnarray*}
 V_t^1(\nabla   g)&:=& {\int \Big (K \la b^{P_s f}_t, \nabla  g \ra+\la b^{P_s f}_t, \nabla \Delta  g \ra\Big) \dm}\\
 &=&  {\int \Big (K \la b^{P_s f}_t, \nabla    g \ra+\la \nabla \div(b^{P_s f}_t), \nabla   g \ra\Big) \dm}
\end{eqnarray*}
and
\begin{eqnarray*}
V_t^2(\nabla  g):= e^{Kt}\int \Big( \frac{\la \nabla  g, \nabla \Delta P_{t+s} f\ra}{|\nabla P_{t+s} f|} -\frac{\la \nabla P_{t+s} f, \nabla  g \ra \la \nabla P_{t+s} f, \nabla \Delta P_{t+s} f\ra}{|\nabla P_{t+s} f|^3}\Big) \dm.
\end{eqnarray*}
From the discussion above we can see that  
\begin{equation}\label{1be2-2}
V_t^1(\nabla P_t g)+V_t^2(\nabla P_t g) = 0,\qquad \forall g\in {\rm F}_t.
\end{equation}

From Lemma \ref{exp1}, we know ${\rm F}_0  \subset  {\rm F}_t$ for any $t\in [0, t_0-s]$. By \eqref{1be2-3} we get
\[
 \Big\{ \nabla P_t g: g\in {\rm F}_0\Big\}\subset \Big\{ \nabla P_t g: g\in {\rm F}_t\Big\} \subset \D_t.
\]
Combining with \eqref{1be2-2} we know
\begin{equation}\label{1be2-5}
V_t^1(\nabla P_t g)+V_t^2(\nabla P_t g)= 0,\qquad \forall  g\in {\rm F}_0.
\end{equation}

Letting $t\to 0$ in \eqref{1be2-5}, by dominated convergence theorem  we obtain
\begin{equation}\label{1be2-6}
V_0^1\big(\nabla  g\big)+V_0^2\big(\nabla g\big)= 0,\qquad \forall  g \in{\rm F}_0.
\end{equation}

By Lemma \ref{exp2} we know  ${\rm F}_0$ includes Lipschitz functions with bounded support. Then by  linearity of  $V_1, V_2$ and an approximation argument (c.f. \cite{G-N}, \cite[Theorem 3.3, \S 4]{Han-JGEA18}),  $V_1, V_2$ can be continuously extended to $$\Big\{g\nabla h: h, g\in {\rm F}_0\Big \} \subset L^2(TX).$$ 
In particular,  we obtain
\begin{equation}\label{1be2-8}
 V_0^1\big(h \nabla \P_s f \big)+  V_0^2\big(h \nabla P_s f\big)= 0,\qquad \forall~h\in {\rm F}_0.
\end{equation}

From the structure of $V_0^2$, we can see that
\begin{eqnarray*}
&&V_0^2(h \nabla  P_s f)\\
&=&{e^{Kt}\int h  \Big( \frac{1}{|\nabla P_{s} f|}\la \nabla   P_{s} f, \nabla \Delta P_{s} f\ra -\frac{1}{|\nabla P_{s} f|^3}\la \nabla P_{s} f, \nabla   P_{s} f \ra \la \nabla P_{s} f, \nabla \Delta P_{s} f\ra\Big) \dm}\\
&=& 0.
\end{eqnarray*}
By \eqref{1be2-8}, for any $h\in {\rm F}_0$, it holds
  
\begin{equation}\label{1be2-1}
 V_0^1\big(h \nabla \P_s f \big)={\int \Big (K \la b_s,   \nabla   P_s f \ra +\la \nabla \div(b_s),  \nabla P_s  f \ra \Big)h \dm}=0.
\end{equation}
By Lemma \ref{exp2}, \eqref{1be2-1} yields 
\[
K \la b_s,   \nabla   P_s f \ra +\la \nabla \div(b_s),  \nabla P_s  f \ra =0.
\]
Hence we can pick $h=\frac 1{|\nabla P_s f|}$ in \eqref{1be2-1}, so that
\[
\int K  |b_s|^2 -\big( \div(b_s)\big)^2 \dm =0.
\]
Note that $|b_s|=e^{Ks}$, it holds
\[
\int \big( \div(b_s)\big)^2 \dm =\int K  |b_s|^2=e^{2Ks} \int K  |b_0|^2=e^{2Ks}\int \big( \div(b_0)\big)^2 \dm
\]
which is the thesis.

\end{proof}

\bigskip

In the following two lemmas, we keep the same notions as in the proof of Proposition \ref{lemma:1be2}.
\begin{lemma}\label{exp1}

For any $r\leq t\leq t_0-s$, we have $ {\rm F}_r \subset {\rm F}_t$.   In particular, ${\rm F}_0$ is an algebra.

\end{lemma}
\begin{proof}

For any $r\leq t\leq t_0-s$ and $g\in {\rm F}_r$, there is $C_2=\big\|\frac{|\nabla P_r g|}{ |\nabla P_{r+s}  f|}\big\|_{L^\infty} >0$ such that
\begin{eqnarray*}
|\nabla P_t g|  & \leq& e^{-K(t-r)} P_{t-r} |\nabla P_r g|\\
&\leq &C_2 e^{-K(t-r)} P_{t-r}  \big( |\nabla P_{r+s}  f|\big)\\
 &=&C_2 |\nabla P_{t+s} f|.
\end{eqnarray*}
Hence $ {\rm F}_r \subset  {\rm F}_t$.

In particular, for any $g, h \in  {\rm F}_0$, there is $C_3>0$ such that
\[
|\nabla (gh)| \leq \|g\|_{L^\infty} |\nabla h|+\|h\|_{L^\infty} |\nabla g| \leq C_3 |\nabla P_s f|,
\]
so by definition $gh\in {\rm F}_0$ and ${\rm F}_0$ is an algebra.

\end{proof}

Next we will  show that the set ${\rm F_0}$ includes all   Lipschitz functions with bounded support.

\begin{lemma}\label{exp2}
The set $\Lip_{bs}(X, \d)$ of  Lipschitz functions with bounded support is a subset of ${\rm F}_0$. In particular, if there is $H\in L^1(X, \mm)$ such that
\[
\int H h\dm,\qquad \forall ~ h\in {\rm F}_0,
\]
then $H=0$.
\end{lemma}
\begin{proof}

 Given $g\in \Lip_{bs}$ with $\supp g \subset  B_R(x)$ for some $R>0$ and $x\in X$. By definition, $|\nabla g| \leq \Lip(g)$ where $\Lip(g)$ is a non-negative real constant.

 By assumption $|\nabla P_s f|=e^{-Ks} P_s|\nabla f|$ and $|\nabla f|\neq 0$. Pick a non-zero non-negative function $G\in L^\infty$ satisfying $G^2 \leq  \min \{|\nabla f|, 1\}$.
So by Lipschitz regularization of the heat flow, $P_s G^2$ is Lipschitz and
\[
P_s G^2\leq P_s |\nabla f| =e^{-Ks} |\nabla P_s f|.
\]
By dimension-free Harnack inequality  \cite[Theorem 3.1]{Li-Harnack},  for any $y_1, y_2\in X$, 
\begin{equation}\label{exp2-1}
\big ((P_s G^2 )(y_1) \big) ^2 \leq  \big ((P_s G )(y_1) \big) ^2 \leq \big(P_s G^2 \big )(y_2)\exp{\Big\{\frac{K\d^2(y_1, y_2)}{e^{2Ks}-1}\Big\}}.
\end{equation}
Let $y_2=x$ in \eqref{exp2-1},  since $G$ is non-zero,  we know $(P_s G^2 )(x) >0$. Let $y_1=x$ and $y_2\in B_R(x)$ \eqref{exp2-1}, 
we know  $\inf_{y\in B_R(x)} P_s G^2>0$. Thus there is $C>0$ such that
\[
|\nabla g| \leq \Lip(g) <C \inf_{y\in B_R(x)} P_s G^2 \leq C e^{-Ks} |\nabla P_s f|\quad \text{on}\quad B_R(x)
\]
which is the thesis.

Furthermore, if 
\[
\int H h\dm,\qquad \forall \quad h\in {\rm F}_0.
\]
Via  approximation by Lipschitz function with bounded support, we can prove that $\int_E H \dm=0$ for all measurable set $E \subset X$. So $H\equiv 0$.
\end{proof}

\bigskip

\begin{theorem}[Equality in the 1-Bakry-\'Emery inequality]\label{thm:localization}
Let $\ms$ be an $\rcd$ probability space with $K\in \R$.  Assume there exists a non-constant $ f\in \V$ attaining the equality in the 1-Bakry-\'Emery inequality
\[
|\nabla P_{t_0} f| =e^{-Kt_0}P_{t_0}|\nabla f|\qquad\text{for some}~~ t_0>0.
\]  
Denote $b_s:=e^{Ks}\frac{\nabla P_s f}{|\nabla P_s f|}$.  Then  the following properties hold:
\begin{itemize}
\item [a)] $\frac{\nabla P_s f}{|\nabla P_s f|}=e^{-Ks}b_s=:b$ is independent of  $s \in (0, t_0)$;
\item [b)] $\nabla \div(b)=-K b$;
\item [c)] $\Delta \div(b)=-K \div(b)$, thus $f=\div(b)$ attains the equality in the 2-Barky-\'Emery inequality.
\end{itemize}

Furthermore, denote  by $(F_t)_{t\in \R^+}$ the regular Lagrangian flow  associated with $ b$,  we have
\begin{equation}\label{local-1}
(F_t)_\sharp \mm=e^{-\frac K2  \big(t^2+\frac2K t \div(b)\big)}\mm \quad \quad  \text{if}~~K\neq 0,
\end{equation}
and
\begin{equation}\label{local-1'}
(F_t)_\sharp \mm=\mm \quad \quad \text{if}~~~~K=0.
\end{equation}

\end{theorem}

\begin{proof}
 
{\bf Part 1}:

By Lemma \ref{lemma:div0} we know $b_s \in {\rm D}(\div)$ for any $s\in (0, t_0)$.
For any $\varphi \in {\rm D}(\Delta)$ and $s,t, h>0$ with $h< \frac12  t$ and $s+t+h<t_0$, we have
\begin{eqnarray*}
 && \int \Big( {P_{t+h}\varphi-P_t \varphi}\Big) \div (b_{t+s})\dm\\
&=& \int \big(P_{t+h} \varphi  \big)\div (b_{t+s+h})\dm-\int \big(P_t \varphi \big) \div (b_{t+s})\dm\\
&&- \int \big(P_{t+h} \varphi \big)  \Big( \div (b_{t+h+s})-\div (b_{t+s})\Big)\dm\\
\text{By Lemma \ref{lemma:1be} }&=& \int \varphi \div(b_s) \dm- \int \varphi \div(b_s) \dm-\int \big(P_{h} \varphi \big)  \Big( \div (b_{h+s})-\div (b_{s})\Big)\dm\\
&=&-\int \big(P_{h} \varphi \big)  \Big( \div (b_{h+s})-\div (b_{s})\Big)\dm.
\end{eqnarray*}
Therefore,
\begin{equation}\label{local-2}
\int \Big(\frac {P_{t+h}\varphi-P_t \varphi}h\Big) \div (b_{t+s})\dm=-\int \big(P_{h} \varphi \big)  \Big(\frac{\div (b_{h+s})-\div (b_{s})}h\Big)\dm.
\end{equation}

By Cauchy-Schwarz inequality and the estimate $\| \Delta P_t \varphi \|_{L^2} \leq \frac 1t \| \varphi \|_{L^2}$ (c.f. Lemma \ref{lemma:reg}), we get the following estimate from \eqref{local-2}
\begin{eqnarray*}
\left |\int P_h \varphi  \Big( \div (b_{h+s})-\div (b_{s})\Big)\dm \right| 
&\leq&\big \| P_{t+h}\varphi-P_t \varphi \big\|_{L^2} \big\| \div (b_{t+s})\big \|_{L^2}\\
&=& \left\| \int_t^{t+h} \Delta P_{s}\varphi \,\d s  \right\|_{L^2} \big\| \div (b_{t+s})\big \|_{L^2}\\
&\leq&  \Big(h\int_t^{t+h} \|\Delta P_{s-h} (P_h \varphi)\|^2_{L^2} \,\d s\Big)^{\frac 12} \, \big\| \div (b_{t+s}) \big \|_{L^2}\\
&\leq&  h\frac 2t \| P_h \varphi\|_{L^2}  \big\| \div (b_{t+s}) \big \|_{L^2}.
\end{eqnarray*}
Thus  by arbitrariness of $\varphi$ and  the density of $P_h\big(L^2(X, \mm)\big)$  in $L^2(X, \mm)$,  we obtain
\[
\big \|  \div (b_{h+s})-\div (b_{s})\big\|_{L^2} \lesssim h.
\]
Therefore  $s\mapsto \div(b_s)$ is absolutely continuous and differentiable in $L^2$ for a.e. $s\in [0, t_0]$.  Furthermore, for  $s\in [0, t_0]$ where $\frac{\d}{\d s} \div(b_s)$ exists, it holds
\begin{eqnarray*}
&&\int (\Delta \varphi)  \div(b_s)\dm\\
\text{By Lemma  \ref{lemma:1be} } &=& \int (\Delta P_t\varphi) \div (b_{t+s})\dm\\
&=& \int (\ddt P_t\varphi) \div (b_{t+s})\dm\\
\text{Letting $h\to 0$ in \eqref{local-2} } &=& -\int \varphi  \frac{\d}{\d s}  \div(b_s)\dm.
\end{eqnarray*}

Therefore, for a.e. $s\in [0, t_0]$, 
 \begin{equation}\label{local-2.3}
 \frac{\d}{\d s}  \div(b_s) =-\Delta  \div(b_s).
 \end{equation}
So by Poincar\'e inequality, we get
\begin{eqnarray*}
\frac{\d}{\d s} \frac12 \int  \big( \div(b_s)\big)^2 \dm &=& \int   \div(b_s) \frac{\d}{\d s}  \div(b_s) \dm\\
\text{By \eqref{local-2.3}}&=&-\int  \div(b_s) \Delta   \div(b_s)\dm\\
&=& \int  |\nabla  \div(b_s)|^2 \dm\\
\text{By Poincar\'e inequality} &\geq& K \int \big(  \div(b_s)\big)^2 \dm.
\end{eqnarray*}
By Gr\"onwall's lemma, we obtain
\begin{equation}\label{local-2.31}
\int \big(  \div(b_s)\big) ^2 \dm \geq e^{2Ks} \int \big ( \div(b_0) \big)^2 \dm.
\end{equation}
 By Proposition \ref{lemma:1be2},  the inequality in \eqref{local-2.31} is actually an equality.
So for  any $s\in (0, t_0)$,    $\div(b_s)$ attains the equality in the Poincar\'e inequality.
 By Lemma \ref{lemma} we know
\begin{equation}\label{local-2.41}
\Delta \div(b_s)=-K \div(b_s).
\end{equation}

For any $\varphi \in \V$, we   have
\[
\int \la \nabla \varphi,  \nabla \div(b_s)\ra=-\int \varphi  \Delta \div(b_s)=\int \varphi K \div(b_s)=\int -K \la b_s, \nabla \varphi \ra.
\]
Thus
\begin{equation}\label{local-2.42}
\nabla  \div(b_s)=-Kb_s.
\end{equation}

In addition,  by \eqref{local-2.3} and \eqref{local-2.41}, it holds $\frac{\d }{\d s}  \div(b_s)=K \div(b_s)$ and
\[
\frac{\d }{\d s} e^{-Ks} \div(b_s)=-K e^{-Ks} \div(b_s)+e^{-Ks}\frac{\d }{\d s}  \div(b_s)=0.
\]
Combining with \eqref{local-2.42}  we know $b:=e^{-Ks} b_s$  is independent of  $s$. 

Finally,  by \eqref{local-2.41} and \eqref{local-2.42} we get
\begin{equation}
\Delta \div(b)=-K \div(b)
\end{equation}
and
\begin{equation}
\nabla  \div(b)=-Kb.
\end{equation}

\bigskip

{\bf Part 2}:

  The identities  \eqref{local-1} and \eqref{local-1'} can be proved using similar argument as \cite[\S 4]{GKKO-R} (and  \cite[\S 2]{ABS19}).  For reader's convenience, we  offer more details here. 

Firstly,  by c) and  Lemma \ref{lemma}, we know $\div(b) \in {\rm TestF}_\loc$ and $\H_{\div(b)}=0$.  Secondly,  by b) and c)  we  know $-K\nabla_{sym} b=\H_{\div(b)}=0$ (c.f. \cite[\S 5]{AT-W} or \cite[\S 3.4]{G-N} for details about the covariant derivative). If $K\neq 0$, $\nabla_{sym} b=0$.  If  $K=0$, by b) it holds
$\nabla \div(b)=0$ so $\div(b)$ is  constant. Note that $\int \div(b)\dm=0$, so $\div(b)=0$. Then  following the argument in \cite[proof of Proposition 2.8]{ABS19} we can still  prove $\nabla_{sym} b=0$.

 Combining  \cite[Theorems 9.7]{AT-W}  of Ambrosio-Trevisan and a truncation argument (c.f. \cite[Theorem 4.2]{GKKO-R}), we  can prove that  the regular Lagrangian flow  $F_t(x)$ associated with $b$ exists for all $(t, x)\in \R^+ \times X$. Thus the curve  $(F_t)_\sharp \mm$ is well-defined for all $t\in \R^+$.

By definition of regular Lagrangian flow $(F_t)$ (c.f. \cite[\S 8]{AT-W}), for any $g\in \V$, $\mu_t=(F_t)_\sharp \mm$ solves the following continuity equation
\begin{equation}\label{translation-1}
\mu_0=\mm,\qquad\ddt \int g \, \d \mu_t=\int b(g) \, \d \mu_t=\int \la b, \nabla g \ra \, \d \mu_t
\end{equation}
for a.e. $t\in \R^+$. It has been proved in \cite[\S 5]{AT-W} that the continuity equation  \eqref{translation-1} has a unique solution.  If $K=0$, it can be seen from $\div(b)=0$ that $\mu_t\equiv \mm$ solves \eqref{translation-1}.  For $K\neq 0$,  we just need to check that $\mu_t:=e^{-\frac K2  \big(t^2+\frac2Kt \div(b)\big)}\mm$   verifies  \eqref{translation-1}.

Given $g\in \V$, by computation, 
\begin{eqnarray*}
&& \ddt \int g \, e^{-\frac K2  \big(t^2+\frac2Kt \div(b)\big)}\d \mm\\ &=&
\int g\big(-Kt- \div(b) \big) \, e^{-\frac K2  \big(t^2+\frac2Kt \div(b)\big)}\d \mm\\
\text{By c)}&=& \int g\Big(-Kt+\frac1K \Delta  \big (\div(b)\big) \Big) \, e^{-\frac K2  \big(t+\frac2Kt \div(b)\big)}\d \mm\\
\text{By b)}&=& \int -Ktg \,e^{-\frac K2  \big(t+\frac2Kt \div(b)\big)} \d \mm+\int    \la b, \nabla g\ra \,e^{-\frac K2  \big(t+\frac2Kt \div(b)\big)} \d \mm\\
&&+\int Ktg |b|^2  \,e^{-\frac K2  \big(t+\frac2K t \div(b)\big)} \d \mm\\
&=&\int \la b, \nabla g \ra \,e^{-\frac K2  \big(t+\frac2K t \div(b)\big)} \d \mm
\end{eqnarray*}
which is the thesis.

\end{proof}

\begin{corollary}\label{coro:negative}
Let $\ms$ be an $\rcd$ {probability}  space with $K\leq 0$.  Then there is no non-constant function attaining  the equality  in the 1-Bakry-\'Emery inequality.
\end{corollary}

\begin{proof}
By  c) of Theorem \ref{thm:localization},  $\Delta \div(b)=-K \div(b)$.  Thus 
\[
0  \leq \int  |\nabla  \div(b)|^2\dm=-\int   \div(b) \Delta \div(b)\dm=K \int  \div(b)^2\dm \leq 0.
\]
So $\div(b)=0$  and $b=0$.
\end{proof}

\bigskip
In the rest of this section we will study the structure of metric measure space, the statements and proofs are almost all taken from the paper of Gigli-Ketterer-Kuwada-Ohta \cite{GKKO-R}.

Let $u$ be a non-constant affine function (c.f.  b)  of  Lemma \ref{lemma}).  We know that  $|\nabla u|$ is  a positive constant and $u$ is Lipschitz.  By  \cite[Theorem 4.4]{GKKO-R} (or \cite[Theorem 3.16]{Han-CVPDE18}),   we know that the gradient flow $(F_t)_{t\geq 0}$ of $u$, which can be seen as a representative of  the regular Lagrangian flow associated with $-\nabla u$ in the sense of Ambrosio-Trevisan \cite[\S 8]{AT-W},   satisfies the following equality (see also \cite{GH-C}) 
\begin{equation}\label{gf-1}
\int \Big (u(x)-u\big(F_t(x)\big)\Big) \dm=\frac12 \int_0^t  \int |\nabla u|^2\circ F_s \dm\, \d s+\frac12 \int_0^t \int |\dot F_s|^2 \circ F_s \dm\,\d s
\end{equation}
and  it induces a family of  isometries   
\begin{equation}\label{gf-2}
\d\big(F_t(x),  F_t(y)\big)=\d(x, y)
\end{equation} for any $x, y \in X, t>0$. More generally, if there is a vector field $b\in L^2(TX)$ with $\div(b) \in L^\infty_\loc$ and $\nabla_{sym} b=0$, by    \cite[Theorem 2.1]{ABS19} (or \cite[Theorem 3.18]{Han-CVPDE18}), the regular Lagrangian flow associated with $b$  induces a family of   isometries.

In particular, there is a decomposition of $X$ in the form $\{X_q\}_{q\in Q}$, where $Q$ is the set of indices, such that $x_0, x_1 \in X_q$  for some $q$ if and only if there is  $t\geq 0$ such that $F_{t}(x_0)=x_1$ or $F_{t}(x_1)=x_0$.  In this case,   $X_q$ is an interval which can be  parametrized by $(F_t)_t$ (or $u$).  Define the  quotient map $\mathfrak Q: X \mapsto Q$ by
\[
q=\mathfrak{Q}(x) \Longleftrightarrow x\in X_q.
\]
There  is a disintegration  of  $\mm$ consistent with $\mathfrak{Q}$ in the following sense.

\begin{definition}[Disintegation on sets, c.f. \cite{AGS-G}, Theorem 5.3.1 and \cite{CavallettiL1}, \S 3.2.3] \label{def:disintegration}
\label{defi:dis}
Let $(X,\mathscr{X},\mm)$ denote a measure space. 
Given any family $\{X_q\}_{q \in Q}$ of subsets of $X$, a \emph{disintegration of $\mm$ on $\{X_q\}_{q \in Q}$} is a measure-space structure 
$(Q,\mathscr{Q},\qq)$ and a map
\[
Q \ni q \longmapsto \mm_{q} \in \mathcal{M}(X,\mathscr{X})
\]
so that:
\begin{enumerate}
\item For $\qq$-a.e. $q \in Q$, $\mm_q$ is concentrated on $X_q$.
\item For all $B \in \mathscr{X}$, the map $q \mapsto \mm_{q}(B)$ is $\qq$-measurable.
\item For all $B \in \mathscr{X}$, $\mm(B) = \int_Q \mm_{q}(B)\, \qq(\d q)$; this is abbreviated by  $\mm = \int_Q \mm_{q} \qq(\d q)$.
\end{enumerate}
\end{definition}

From Theorem \ref{thm:localization} and Lemma \ref{lemma}, we know  there is a  decomposition $\{X_{q}\}_{q \in Q} $ induced by   $b$ (or $-\frac 1K \nabla \div(b)$ when $K>0$) satisfying the following properties.

\begin{corollary}\label{coro:local}
Keep  the same assumptions  and notations as in Theorem \ref{thm:localization}, assume further that $K>0$.  Then  there exists a decomposition  $\{X_{q}\}_{q \in Q} $ of $ X$ induced by the regular Lagrangian flow  $(F_t)$ associated with $b$,  such that: 
\begin{enumerate}
\item   for any $q \in \Q$, $X_q$ is a  geodesic line  in $(X,\d)$;
\item   for any $q \in \Q$, $x_1, x_2 \in X_q$, there is a unique $t$ such that 
$$t=t|b|=\d(x_1, x_2).$$ and  $F_{t}(x_0)=x_1$ or $F_{t}(x_1)=x_0$;
\item  there exists a disintegration of $\mm$ on $\{X_{q}\}_{q \in Q}$
\[
\mm= \int_{Q} \mm_{q} \, \qq(\d q), \qquad \qq(Q) = 1;
\]
\item  for $\qq$-a.e. $q \in Q$ and any $t>0$, it holds  $$(F_t)_\sharp \mm_q=e^{-\frac K2  \big(t^2+\frac 2K t \div(b)\big)}\mm_q,$$ 
and the 1-dimensional metric measure space $(X_{q}, \d,\mm_{q})$ satisfies $\cd$;
\item for $\qq$-a.e. $q \in Q$,
 $\div(b)\restr{X_q}$  can be represented by
\[
  \div(b)(x)={\rm sign}\big( \div(b)(x)\big)K\d(x, x_q),~~~~x\in X_q,
\]
where $x_q$ is the unique point in $X_q$ such that $\div(b)(x_q)=0$. In particular,
\begin{equation}\label{coro:local1}
\int \div(b) \, \d \mm_q=0,\qquad \qq-\text{a.e.}~q\in Q.
\end{equation}
\end{enumerate}
\end{corollary}

\begin{proof}
From the construction of the decomposition discussed  before, it is not hard to see the  validity of assertions (1)- (3) which are actually a variant of measure-decomposition theorem (see also \cite{CavallettiL1}). Assertion (4) is a consequence of \eqref{local-1} in Theorem \ref{thm:localization}. We  will  just prove   (5).
For $u:=\frac 1K \div(b)$,   by \eqref{gf-1} and   Lemma \ref{lemma:lipbound}  below  we have
\begin{eqnarray*}
&&\int_Q \int_{X_q} \Big (u(x)-u\big(F_t(x)\big)\Big) \, \d \mm_q\, \d \qq(q) \\&=&
\int \Big (u(x)-u\big(F_t(x)\big)\Big) \dm\\
&=&\frac12\int_0^t \int |\nabla u|^2\circ F_s \dm\, \d s+\frac12 \int_0^t \int |\dot F_s|^2 \circ F_s \dm\,\d s\\
&\geq& \frac12 \int_Q \Big ( \int_0^t \int_{X_q}  |{\rm lip} (u\restr{X_q})|^2\circ F_s\,\d \mm_q \,\d s+\frac12 \int_0^t \int |\dot F_s|^2\circ F_s \dm_q\,\d s \Big) \,  \d \qq(q).
\end{eqnarray*}
Thus for a.e. $q\in Q$,   $X_q$ is the trajectories of the gradient flow of   $u=\frac 1K \div(b)$:
$$\left|\frac 1K  \div(b)(x_1)-\frac 1K  \div(b)(x_2) \right |=\frac 1K  |\nabla  \div(b) |\d(x_1, x_2)=\d(x_1, x_2),~~~\forall x_1, x_2 \in X_q.$$
As $u$ is non-constant,  there  is  a unique point $x_q \in X_q$ such that $\div(b)(x_q)=0$. So  $\div(b)$  can be represented by
\[
  \div(b)(x)={\rm sign}\big( \div(b)(x)\big)K\d(x, x_q),\qquad \forall x\in X_q,
\]

\end{proof}

\begin{lemma}\label{lemma:lipbound}
For any $g\in \V\cap \Lip(X, \d)$ and $s\in [0, t]$,  the following inequality holds
\begin{equation}\label{lipbound-1}
\int |\nabla  g|^2 \circ F_s \dm  \geq \int_Q \int_{X_q}  |{\rm lip} (g\restr{X_q})|^2\circ F_s\,\d \mm_q  \d \qq(q).
\end{equation}

\end{lemma}

\begin{proof}
Let $(g_n)_n \subset  L^2$ be a sequence of Lipschitz functions such that $g_n \to g$ and $|\lip{g_n}| \to |\nabla g|$ in $L^2(X, (F_s)_\sharp \mm)$.
Note that $ (F_s)_\sharp \mm=\int_Q \big( (F_s)_\sharp \mm_q \big) \,\d\qq(q)$, there is a subsequence of $(g_n)$, still denoted by $(g_n)$, such that
$g_n\restr{X_q} \to g\restr{X_q}$  in $L^2(X_q, (F_s)_\sharp\mm_q)$ for $\qq$-a.e. $q\in Q$.

Notice that $|\lip{g_n}|\restr{X_q} \geq |{\rm lip}(g_n\restr{X_q})|$, and it is known that $ |{\rm lip} (g\restr{X_q})|=|\nabla g\restr{X_q}|$ $\mm_q$-a.e. on $X_q$ (since  the values of local Lipschitz constant and weak upper gradient are independent of  (locally Lipschitz) weighted measures).  Then we have
\begin{eqnarray*}
\int |\nabla  g|^2\circ F_s \dm  &=& \int |\nabla  g|^2\, \d (F_s)_\sharp \mm\\
&=& \lmt{n}{\infty} \int |\lip{g_n}|^2 \, \d (F_s)_\sharp \mm\\&=&
 \lmt{n}{\infty} \int_Q \Big(\int_{X_q} |\lip{g_n}|^2\,\d (F_s)_\sharp \mm_q  \Big)\d \qq(q)\\
 \text{By Fatou's lemma} &\geq&   \int_Q \lmt{n}{\infty} \Big(\int_{X_q} |\lip{g_n}|^2\,\d (F_s)_\sharp\mm_q  \Big)\d \qq(q)\\
 &\geq&  \int_Q \lmt{n}{\infty} \Big(\int_{X_q} |{\rm lip}(g_n\restr{X_q})|^2\,\d (F_s)_\sharp\mm_q  \Big)\d \qq(q)\\
\text{By definition of the energy form $\E$} &\geq& 
 \int_Q  \Big(\int_{X_q} |\nabla g\restr{X_q}|^2\,\d (F_s)_\sharp\mm_q  \Big)\d \qq(q)\\
 &=&\int_Q \int_{X_q}  |{\rm lip} (g\restr{X_q})|^2\,\d (F_s)_\sharp \mm_q  \d \qq(q)\\
  &=&\int_Q \int_{X_q}  |{\rm lip} (g\restr{X_q})|^2\circ F_s \,\d \mm_q  \d \qq(q)
\end{eqnarray*}
which is the thesis.
\end{proof}

\begin{remark}\label{remark:lipbound}
Unlike the well-known result of Cheeger \cite[Theorem 6.1]{C-D} which tells us  that $|\nabla g|=|\lip{g}|$ $\mm$-a.e. if $\ms$ satisfies volume doubling property and supports a local Poincar\'e inequality, it is still unknown whether this result is still true on $\rcd$ spaces or not.  In \cite{GH-I}, the author and Gigli prove that $|\nabla g|_p=|\nabla g|$ for all $p>1$ on $\rcd$ spaces. But it is still possible that $|\nabla g|<|\lip{g}|$. 
\end{remark}
\subsection{Proof of the rigidity}\label{sec:proofbe}
In this part, we will complete the proof of Theorem \ref{th3-intro} by proving the following Proposition \ref{thm:berigidity2}, \ref{thm:berigidity}.

In  Proposition \ref{prop:1dimrigidity}, we proved the rigidity of the 2-Bakry-\'Emery inequality for 1-dimensional spaces.  Generally,  it is   proved by Gigli-Ketterer-Kuwada-Ohta \cite{GKKO-R}  that  $\ms$ is  isometric to the product space of the 1-dimensional Gaussian space and an $\rcd$ space, if there is a non-constant function attaining the equality in the Poincar\'e inequality.
As a consequence of Theorem \ref{thm:localization}, Lemma \ref{lemma} and the result of Gigli-Ketterer-Kuwada-Ohta,   we get  the following proposition.

\begin{proposition}[c.f. \cite{GKKO-R}, Theorem 1.1]\label{thm:berigidity2}
Let $\ms$ be an $\rcd$ space with $K>0$. Assume there is a non-constant $f\in \V$ attaining the equality in the 1-Bakry-\'Emery inequality.   Then  there exists an $\rcd$-space $(Y, \d_Y, \mm_Y)$,  such that
the metric space $\ms$ is isometric to the product space $$  \Big(\R, | \cdot |, \sqrt{{K}/(2\pi)} \exp(-{Kt^2}/2)\mathcal \d t \Big) \times (Y, \d_Y, \mm_Y)$$ equipped with the $L^2$-product metric and product measure.

\end{proposition}
\begin{proof}[Sketch of the proof]
By (c) of Theorem \ref{thm:localization} and Lemma \ref{lemma},   $u=\frac 1K \div\Big(\frac{\nabla P_t f}{|\nabla P_t f|}\Big)$ attains the equality in the Poincar\'e inequality.  Then the assertion follows from \cite[Theorem 1.1]{GKKO-R}.

For reader's convenience, we  offer more details here. 
By Theorem \ref{thm:localization} and Lemma \ref{lemma}, $\H_u=0$ and  $|\nabla u|=1$, so that $-\nabla u$ induces  a  family of isometries $(F_t)$.  By Corollary \ref{coro:local}, there is a disintegration $\mm=\mm_q \qq(\d q)$ associated with the one-to-one map  $\Psi:  \R \times u^{-1}(0)\ni (r, x)    \mapsto F_r(x) \in X$.

 In addition,  assume (in the coordinate of $\Psi$) that $u\big( (0, y) \big)=0$.    By  (4) and (5) of Corollary \ref{coro:local},  up to a reflection, we may write 
 \[
 u\big( (r, y) \big)=r,
 \]
 and
 \[
 (F_r)_\sharp \mm_q=e^{-\frac K2  \big(r^2+ 2 u r \big)}\mm_q.
 \]
 Hence $\mm_q \ll \mathcal H^1 \restr{X_q}$ with continuous density $h_q$,  and 
 \[
h_q\big( (r, y) \big)=e^{-\frac K2  \big(r^2+ 2 u( (0, y) ) r \big)}h_q\big( (0, y) \big)=e^{-\frac {Kr^2} 2}h_q\big( (0, y) \big).
 \]
 So $\mm$ is  isomorphic to a product measure $ \Phi_K \times \mm_Y$. 
 
 Following  Gigli's strategy of the splitting theorem \cite{G-S},   one can prove that  the map $\Psi$ induces an isometry between the Sobolev spaces $W^{1,2}\big(\Psi^{-1}(X)\big)$ and $W^{1,2}\big(\R \times u^{-1}(0)\big)$.  Then from Sobolev-to-Lipschitz property we know that  $\Psi$ is an isometry  between metric measure spaces (see \cite[\S 6]{G-S}, \cite{G-Overview}, and  \cite[\S 5]{GKKO-R} for details).
\end{proof}

\bigskip

Finally, we have the following characterization of  extreme functions.

\begin{proposition}\label{thm:berigidity}
Under the same assumption  and keep the same notations as Proposition \ref{thm:berigidity2},    $ f$ can be  represented  in the coordinate of the product space $\R \times Y$,  by $$  f(r, y)=\int_0^r g(s)\,\d s,\qquad (r, y) \in \R \times Y$$
for some non-negative $g\in L^2(\R, \phi_K \mathcal L^1)$.
In particular, if $f$ attains the equality in the 2-Bakry-\'Emery inequality,   then $P_t f(r, y)=C e^{Kt}  r$ for some constant $C$.

\end{proposition}

\begin{proof}
By Theorem \ref{thm:localization} and the proof of  Proposition \ref{thm:berigidity2}, we know
\begin{equation}\label{1-ber}
\frac{\nabla  f}{|\nabla f|}=\nabla \frac 1K \div \Big(\frac{\nabla  f}{|\nabla f|}\Big)=\nabla r.
\end{equation}   So for $\mm_Y$-a.e. $y\in Y$,  $$ f(r, y)-f(0, y)=\int_0^r |\nabla f|(s, y) \, \d s.$$
Given $r\in \R$, from \eqref{1-ber} we can see that $f(r, y)$ is independent of $y\in Y$, so we can assume $f(0, y)=0$ and  denote  $g(s):=|\nabla f|(s, y)$ which is the thesis.

   If $f$  also attains the equality in the 2-Bakry-\'Emery inequality, by Lemma \ref{lemma:lipbound}, \eqref{coro:local1},  and a standard localization argument we can see that $f (\cdot, y)$ attains the quality in the 1-dimensional  Poincar\'e  inequality for $\mm_Y$-a.e. $y\in Y$. Then the  second assertion follows from Proposition \ref{prop:1dimrigidity}.
\end{proof}

\section{Rigidity of some  functional  inequalities}\label{sec:funct}

\subsection{Equality in Bobkov's inequality}
 In this part,  we will study the cases of  equality  in  Bobkov's inequality, as well as the  Gaussian isoperimetric inequality,  and prove the corresponding  rigidity theorems,

Using an  argument   of Carlen-Kerce  \cite[Section 2]{CarlenKerce01}   (which was  firstly used  by Ledoux in \cite{Ledoux98},   see also a recent work of Bouyrie \cite{Bouyrie17}), we can prove  the following monotonicity formula concerning $\rcd$ spaces for $K > 0$. 

\begin{proposition}\label{prop1:monotone}
Let  $\ms$ be  a  $\rcd$ space with  $K> 0$. For  any $f: X \mapsto [0, 1]$, $t>0$,  denote $f_t=P_t f$ and define  
\begin{equation}\label{eq1:prop1:monotone}
J_K(f_t):= \int \sqrt{I_K(f_t)^2+\Gamma( f_t)}\,\d \mm
\end{equation}
where $I_K$ is the Gaussian isoperimetric profile defined in \eqref{intro:GII-1}.

Then  for $\mathcal L^1$-a.e. $t$, we have 
\begin{eqnarray*}
&&\ddt J_K(f_t)\\
 &=&  -\int G_K^{-\frac32} \Big ( \big\| I_K\H_{f_t}-I_K' \nabla f_t \otimes \nabla f_t \big \|^2_{\rm HS} +\| \H_{f_t}\|_{\rm HS}^2\Gamma(f_t)
 -\frac14 \Gamma\big(\Gamma(f_t)\big) \Big )\,\d \mm\\&&-\int G_K^{-\frac12} \Big(\d\bRic(f_t, f_t)-K\Gamma(f_t)\dm\Big)
\end{eqnarray*}
where $G_K=I_K(f_t)^2 +\Gamma(f_t)$.

 In particular, $J_K(f_t)$ is non-increasing in $t$. 
\end{proposition}
\begin{proof}

If $f$ is constant, $J_K(f_t)$ is also a constant function of $t$, there is nothing to prove.  So we  assume that $f$ is not  constant. In addition, similar to  \cite[Proof of Theorem 3.1, Step 1]{AmbrosioMondino-Gaussian}, it suffices to prove the assertion for every $f\in \Lip(X, \d)$ taking values in $[\epsilon, 1-\epsilon]$, for some $\epsilon \in (0, \frac12)$. In fact,  for general $f$, we can replace $f$ by $f^\epsilon:=\frac 1{1+2\epsilon} (f+\epsilon)$, then letting $\epsilon \downarrow 0$ we will get the answer.

 It is known that $ f_t \in L^\infty (X, \mm) \cap {\rm D}(\Delta)$, and $\Delta f_t \in \V$. By Lipschitz regularization of $P_t$ (c.f. \cite[Theorem 6.5]{AGS-M}), we also have $f_t \in \Lip(X, \d)$ for any $ t\in (0, \infty)$, so $f_t \in {\rm TestF}$.  From \cite[Lemma 3.2]{AmbrosioMondino-Gaussian} we know $t \mapsto J_K(f_t)$ is Lipschitz, and  for $\mathcal L^1$-a.e. $t$ we have
\begin{eqnarray*}
\frac{\d J_K}{\d t}&=&\ddt  \int \sqrt{G_K(f_t)}\, \d \mm \\
&=& \int G_K(f_t)^{-\frac12} \Big(  I_K(f_t)I_K'(f_t) \Delta f_t+\Gamma(f_t, \Delta f_t) \Big)\,\d \mm,
\end{eqnarray*}
where  $G_K(f)$ denotes the function $I_K(f)^2 + \Gamma(f)$. 
Notice that  by  minimal (maximal) principle,   $G_K(f_t) >\delta$ for some $\delta>0$. Thus the  formula above is well-posed.

From  the definition of $\bRic$  in Proposition \ref{prop:measurebochner},  we can see that 
\begin{eqnarray*}
\frac{\d J_K}{\d t}
&=& \underbrace{\int  G_K^{-\frac12} I_K(f_t)I_K'(f_t) \Delta f_t \,\d \mm}_{J_1}-\underbrace{\int \frac12 \Gamma\big (G_K^{-\frac12}, \Gamma(f_t)\big)+G_K^{-\frac12}\Big(\| \H_{f_t}\|^2_{\rm HS}+K\Gamma(f_t) \Big)\,\d \mm}_{J_2}\\
&&-\int G_K^{-\frac12} \Big(\d\bRic(f_t, f_t)-K\Gamma(f_t)\dm\Big).
\end{eqnarray*}
Notice that $G_K$ admits a quasi continuous
representative, so we can integrate it with respect to the measure-valued Ricci
tensor.

Thus  the non-smooth Bochner inequality in  Proposition \ref{prop:measurebochner} yields
\begin{equation}\label{prop1:eq1}
\frac{\d J_K}{\d t} \leq J_1+J_2.
\end{equation}

By computation,
\begin{eqnarray*}
J_1&=& -\int \Gamma\big(G_K^{-\frac12} I_K I_K', f_t\big)\,\d \mm\\
&=&  -\int G_K^{-\frac12} (I_K I_K')' \Gamma(f_t)\,\d \mm+\frac12  \int G_K^{-\frac32} I_K I_K' \Gamma(G_K, f_t)\,\d \mm\\
&=&-\int G_K^{-\frac12} (I_K I_K')' \Gamma(f_t)\,\d \mm\\
&&+\frac12  \int G^{-\frac32} I_K I_K' \Big (2I_KI_K' \Gamma(f_t, f_t)+\Gamma\big(\Gamma(f_t), f_t\big)\Big )\,\d \mm\ \\
&=&  -\int G_K^{-\frac12}\big( (I'_K)^2-K \big ) \Gamma(f_t)\,\d \mm+  \int G_K^{-\frac32} (I_K I_K')^2 \Gamma(f_t)\,\d \mm \\
&&~~~~~~+  \int G_K^{-\frac32} I_K I_K' \H_{f_t}( f_t,  f_t)\,\d \mm\\
&=&   -\int G_K^{-\frac32} \Big ( (I_K')^2 \Gamma(f_t)^2\underbrace{-KI_K^2 \Gamma(f_t)-K\Gamma(f_t)^2}_{=-K\Gamma(f_t)G_K(f_t)}-I_KI_K'\H_{f_t}(f_t, f_t)\Big )\,\d \mm
\end{eqnarray*}
where  in the fourth equality  we use  the  identity   $(I _KI_K')'=(I_K')^2-K$ which follows from $I_K I_K''=-K$.

Similarly, 
\begin{eqnarray*}
&&- \frac12 \int \Gamma \big( G_K^{-\frac12}, \Gamma(f_t)\big)\,\d \mm\\
 &=&\int \frac14 G_K^{-\frac32}\Big ( 2I_KI_K'\Gamma\big(f_t, \Gamma(f_t)\big)+\Gamma\big(\Gamma(f_t)\big)\Big )\,\d \mm\\
 &=&  \int  G_K^{-\frac32}\Big ( I_KI_K'\H_{f_t}(f_t, f_t )+ \frac14  \Gamma\big(\Gamma(f_t)\big) \Big )\,\d \mm.
\end{eqnarray*}

In summary, we get
\begin{eqnarray*}
&&J_1+J_2 \\
&=&   -\int G_K^{-\frac32} \Big ( (I_K')^2 \Gamma(f_t)^2 -\frac14 \Gamma\big(\Gamma(f_t)\big)
-2I_K I_K'\H_{f_t}(f_t, f_t)+\| \H_{f_t}\|_{\rm HS}^2\big (I_K^2+\Gamma(f_t)\big) \Big)\,\d \mm\\
&= &  -\int G_K^{-\frac32} \Big ( \big\|I_K\H_{f_t}-I_K' \nabla f_t \otimes \nabla f_t \big \|^2_{\rm HS} +\| \H_{f_t}\|_{\rm HS}^2\Gamma(f_t) 
 -\frac14 \Gamma\big(\Gamma(f_t)\big) \Big )\,\d \mm.
\end{eqnarray*}

Recall that by definition
\begin{eqnarray*}
 \Gamma\big(\Gamma(f_t)\big)&=&2\H_{f_t} \big(\nabla  f_t, \nabla  \Gamma(f_t) \big)\\
&\leq&2 \| \H_{f_t} \|_{\rm HS}\sqrt{\Gamma(f_t)} \sqrt{ \Gamma\big(\Gamma(f_t)\big)},
\end{eqnarray*}
thus
$$\| \H_{f_t}\|^2_{\rm HS}\Gamma(f_t) 
 \geq \frac14 \Gamma\big(\Gamma(f_t)\big).$$ 
 
Combining with \eqref{prop1:eq1} we have $$\frac{\d J_K}{\d t}  \leq J_1+J_2 \leq 0,$$ so $t\mapsto J_K(f_t)$ is  non-increasing.
\end{proof}

Appying  Proposition \ref{prop1:monotone},   we  obtain the  functional version of Gaussian  isoperimetric inequality of Bobkov on $\rcd$ spaces, which had been  proved  by Ambrosio-Mondino in \cite{AmbrosioMondino-Gaussian} using a different proof (see  also  \cite[Chapter 8.5.2]{BakryGentilLedoux14} for more discussions).
\begin{proposition}
Let  $\ms$ be a metric measure space satisfying  $\rcd$ condition for some $K > 0$. Then $\ms$ supports $K$-Bobkov's isoperimetric inequality in the sense of Definition \ref{def:bobkov},
\[
 I_K \left (\int f \,\d \mm \right )  \leq J_K(f)
\]
for all measurable function $f$ with values in $[0, 1]$.
\end{proposition}
\begin{proof}
Let $f$ be a measurable function with values in $[0, 1]$.  By Proposition \ref{prop1:monotone}  and definition of $J_K(f)$ we know
\[
\lmts{t}{+\infty} J_K(f_t)  \leq \lmti{t}{0}J_K(f_t)=J_K(f).
\]
Combining with the  ergodicity of heat flow and the 2-Bakry-\'Emery inequality
\[
\lmt{t}{+\infty} J_K(f_t) =I_K \left (\int f \,\d \mm \right ),
\]
we get  Bobkov's isoperimetric inequality.
\end{proof}

In the next proposition, we discover  the cases of equality in  Bobkov's inequality. By Proposition \ref{prop1:monotone},  we simultaneously  obtain the rigidity of the Gaussian isoperimetric  inequality.   We refer the readers  to \cite[Section 2]{CarlenKerce01} for related discussions on $\R^n$.

\begin{proposition}[Equality in  Bobkov's inequality]\label{prop2:splitting}
Let $\ms$ be a  $\rcd$ metric measure space with $K>0$. Then  there exists a non-constant $f$ attaining the equality $ I_K \left (\int f \,\d \mm \right ) =J_K(f)$   if and only if
 $$\ms \cong  \Big (\R, | \cdot |, \sqrt{{K}/(2\pi)} e^{-{Kt^2}/2}\mathcal \d t\Big)\times (Y, \d_Y, \mm_Y) $$ for some  $\rcd$ space $(Y, \d_Y, \mm_Y) $,
and   up to change of variables, $f$ is either  the indicator function of a half space 
\[
f(r, y)=\chi_E,\qquad E={ (-\infty, e] \times Y},
\]
where $e\in \R\cup \{+\infty\}$ with $\int_{-\infty}^e \phi_K(s)\,\d s=\int f\dm$;
or else, there are  $a=(2\int f)^{-1}$ and  $b=\Phi^{-1}_K\big(f(0, y)\big)$ such that
$$f(y, t)=\Phi_K(at+b)=\int_{-\infty}^{at+b} \phi_K(s)\,\d s.$$

\end{proposition}

\begin{proof}

{\bf Part 1}: Denote $f_t=P_t f$  and $h_t=\Phi_K^{-1}( f_t)$. We will show that $h_t$ satisfies  ${\bf \Gamma}_2(h_t)=K\Gamma(h_t)\,\mm$ (c.f.  Proposition \ref{prop:measurebochner}), and thus satisfies  (1)  in Lemma  \ref{lemma}.

By Proposition \ref{prop1:monotone} we know  $ I_K \left (\int f \,\d \mm \right ) =J_K(f)$ if and only if  $$ I_K \left (\int f \,\d \mm \right ) =I_K \left (\int f_t \,\d \mm \right ) =J_K(f_t)\qquad \text{for all}~~~ t\geq 0,$$ which is equivalent to  $\frac{\d J_K}{\d t}=0$ for all $t>0$. From  Proposition \ref{prop1:monotone}, we  know that $\frac{\d J_K}{\d t}=0$ if and only if  the following equalities \eqref{prop2:eq-1} \eqref{prop2:eq-2} \eqref{prop2:eq01} are satisfied
\begin{equation}\label{prop2:eq-1}
\bRic(f_t, f_t)=K\Gamma(f_t)\,\mm,
\end{equation}
\begin{equation}\label{prop2:eq-2}
I_K\H_{f_t}-I_K' \nabla f_t \otimes \nabla f_t =0
\end{equation}
and 
\begin{equation}\label{prop2:eq01}
\| \H_{f_t}\|^2_{\rm HS} \Gamma(f_t) - \frac14 \Gamma\big(\Gamma(f_t)\big)=0.
\end{equation}

\bigskip

By definitions,
\begin{equation}\label{prop2:eq1}
I_K(f_t) = \phi_K(h_t),\qquad f_t = \Phi_K(h_t).
\end{equation}
By \eqref{prop2:eq1} and chain rule (c.f. \cite[Theorem 2.2.6]{G-N}), and the fact that the vector fields are fully supported (c.f. Lemma \ref{exp4}), we get
\[
I'_K(f_t)\nabla f_t = -h_t\phi_K (h_t)\nabla h_t,\qquad\nabla f_t = \phi_K(h_t)\nabla h_t.
\]
Then we have  $$I'_K(f_t)=-h_t,\qquad\nabla f_t = I_K(f_t) \nabla h_t,$$ and
\[
\H_{f_t}= -h_t\phi_K(h_t)\nabla h_t \otimes \nabla h_t + \phi_K(h_t)\H_{h_t}.
\]

In conclusion,  we obtain
\begin{equation}\label{eq0:bob}
 \nabla h_t=I_K^{-1}(f_t)\nabla f_t 
\end{equation}
and
\begin{equation}\label{eq1:bob}
\H_{f_t}= I'_K(f_t)I^{-1}_K(f_t)\nabla f_t \otimes \nabla f_t+ I_K(f_t)\H_{h_t}.
\end{equation}

\bigskip

 By \eqref{eq0:bob} and the  bi-linearity  of $\bRic(\cdot, \cdot)$, \eqref{prop2:eq-1}  is equivalent to
\begin{equation}\label{prop2:eq-1'}
\bRic(h_t, h_t) =K\Gamma(h_t)\,\mm.
\end{equation}

Comparing   \eqref{eq1:bob} and \eqref{prop2:eq-2}, we can see that $f_t$ satisfies \eqref{prop2:eq-2} if and only if $\H_{h_t} = 0$, which is equivalent to
\begin{equation}\label{eq2:bob}
\|\H_{h_t}\| _{\rm HS}= 0.
\end{equation}

By \eqref{eq1:bob} and \eqref{eq2:bob}, we have
\[
\| \H_{f_t}\|_{\rm HS} =\|I'_KI^{-1}_K\nabla f_t \otimes \nabla f_t\|_{\rm HS}= I'_KI^{-1}_K \Gamma(f_t)
\]
and
\begin{eqnarray*}
 \Gamma\big(\Gamma(f_t)\big)&=&2\H_{f_t}\big(\nabla f_t, \nabla \Gamma(f_t)\big)\\
\text{By \eqref{eq1:bob}}\qquad&=&2 I'_KI^{-1}_K\Gamma( f_t) \Gamma\big( f_t, \Gamma(f_t)\big)\\
&=& 4 I'_KI^{-1}_K\Gamma( f_t) \H_{f_t}(\nabla f_t, \nabla f_t )\\
\text{By \eqref{eq1:bob}}\qquad&=& 4 I'_KI^{-1}_K\Gamma( f_t) \Big(I'_K I_K^{-1} \big(\Gamma(f_t)\big)^2 \Big).
\end{eqnarray*}
Therefore,
\begin{equation}\label{prop2:eq01'}
\| \H_{f_t}\|^2_{\rm HS}\Gamma(f_t) - \frac14 \Gamma\big(\Gamma(f_t)\big)=( I'_KI^{-1}_K)^2 \big(\Gamma(f_t)\big)^3-( I'_KI^{-1}_K)^2 \big(\Gamma(f_t)\big)^3=0
\end{equation}
which is exactly \eqref{prop2:eq01}.

In conclusion,  \eqref{prop2:eq-1} \eqref{prop2:eq-2} \eqref{prop2:eq01} $\Longleftrightarrow$ \eqref{prop2:eq-1'}  \eqref{eq2:bob},   and the latter ones are equivalent to
\begin{equation}\label{prop2:eq3}
{\bf \Gamma}_2(h_t)=K\Gamma(h_t)\,\mm
\end{equation}
 which is the thesis.

\bigskip

{\bf Part 2}: By Proposition  \ref{thm:berigidity} we just need to study  the 1-dimensional cases. 
By Proposition  \ref{prop:1dimrigidity}  we  know $h_t=\Phi_K^{-1}( f_t)$ is an affine  function on $\R$  for any $t>0$,   there exist   $a=a(t), b=b(t) \in \R$ such that 
$$f_t(x)=\Phi_K(ax+b)=\int_{-\infty}^{ax+b} \phi_K(s)\,\d s.
$$

By \cite[Theorem 1]{CarlenKerce01} there is  $s\geq 0$ such that $$f_t=P_{t+s}(\chi_E),~~~\forall~t\geq 0$$
where $E$  is the half-line such that $\int_E \phi_K \, \d \mathcal L^1=\int f\dm$.

Therefore, if $s=0$, $f=\chi_E$. Otherwise, $a(t), b(t)$ are continuous on $[0, +\infty)$, so $$f=\Phi_K(a_0x+b_0)=\int_{-\infty}^{a_0x+b_0} \phi_K(s)\,\d s$$
where $a_0=(2\int f)^{-1}$, $b_0=\Phi^{-1}_K(f(0))$. 
\end{proof}

Applying Proposition \ref{prop2:splitting}, we obtain the  rigidity of the Gaussian isoperimetric inequality.
\begin{corollary}[Rigidity of the Gaussian isoperimetric inequality]\label{coro:gaussian}
Let $\ms$ be a  $\rcd$ metric measure space with $K>0$. 
If there is a Borel set $E\subset X$ with positive $\mm$-measure  such that 
\[
P(E)=J_K(\chi_E) =I_K\big(\mm(E)\big).
\]
Then  $$\ms \cong   \Big(\R, | \cdot |, \sqrt{{K}/(2\pi)} e^{-{Kt^2}/2}\mathcal \d t \Big) \times (Y, \d_Y, \mm_Y)$$ for some  $\rcd$ space $(Y, \d_Y, \mm_Y) $, and
$E\cong {(-\infty, e]} \times Y$ with $e=\Phi^{-1}_K\big(\mm(E)\big)$.
\end{corollary}

\subsection{Equalities  in $\Phi$-entropy inequalities} \label{sect:phi}
In this part we will characterize the cases of equalities in the logarithmic Sobolev inequality,  the Poincar\'e inequality,   and more generally,  $\Phi$-entropy inequalities of Chafa\"i \cite{Chafai04} and Bolley-Gentil \cite{BolleyGentil10} on $\rcd$ metric measure spaces.

First of all, we prove a general $\Phi$-entropy inequality. For more discussions  about  admissible $\Phi$'s, see \cite[Page 330]{Chafai04}, \cite[Section 1.3]{BolleyGentil10} and the references therein.

\begin{proposition}\label{prop:phi}
Let  $\ms$ be a  metric measure space satisfying  $\rcd$ condition for some $K>0$. Let $\Phi$ be a $C^2$-continuous  strictly convex function on an interval $I\subset \R$ such that  $\frac{1}{\Phi''}$
is concave.  Then $\ms$ satisfies the following $\Phi$-entropy inequality:
\begin{equation}\label{eq1-prop:phi}
\underbrace{\ent{\mm}^\Phi(f)}_{:=\int \Phi(f)\dm}-\Phi\Big(\int f\dm \Big) \leq \frac1{2K}\int \Phi''(f) \Gamma(f)\dm
\end{equation}
for all  $I$-valued functions $f$.
\end{proposition}
\begin{proof}
Let $f$ be an $I$-valued function and denote $f_t:=P_t f$.
By the  ergodicity of the heat flow, we have
\begin{eqnarray*}
\ent{\mm}^\Phi(f)-\Phi\Big(\int f\dm \Big)&=&-\int_0^{+\infty} \ddt \ent{\mm}^\Phi(f_t)\,\d t\\
\text{By~\cite[Theorem 4.16]{AGS-C} }&=& \int_0^{+\infty}  \int  \Phi''(f_t) \Gamma(f_t)\dm\, \d t\\
\text{By~\eqref{eq0:lemma:phi}, Proposition \ref{theorem:phi}} &\leq & \int_0^{+\infty} e^{-2Kt} \int P_t\big( \Phi''(f) \Gamma(f)\big)\dm \, \d t\\
&=& \frac1{2K} \int  \Phi''(f) \Gamma(f)\dm 
\end{eqnarray*}
which is the thesis.
\end{proof}

\bigskip

Finally, we   complete the proof of  Theorem \ref{th2-intro}.

\begin{proof}[Proof of Theorem \ref{th2-intro}]
We keep the same notations as in $\S$\ref{section2:intro}.  If there is a function $f$ attaining the equality in \eqref{eq1-prop:phi}, from the proof of Proposition \ref{prop:phi}, we can see that
\[
{\Phi''(P_t f)}{ {\Gamma(P_t f )}}   =  e^{-2Kt} P_t \big({\Phi''}(f){ \Gamma(f)} \big)
\]
for almost every $t>0$.
If $f$ is not constant, by Proposition \ref{theorem:phi} (or Corollary \ref{coro:phi})  and Proposition \ref{thm:berigidity2} we know $\ms$ is isometric  to the product  $\big (\R, | \cdot |, \phi_K\mathcal L^1\big) \times (Y, \d_Y, \mm_Y)$ of two $\rcd$ metric measure spaces.
Concerning the extreme functions, by Corollary \ref{coro:phi}  and Proposition \ref{thm:berigidity} we just need to consider the following two cases

\begin{itemize}
\item [a)]  Poincar\'e inequality:  $\Phi=x^2$ for $x\in\R$.  If  there is a non-constant   function $f\in \V$ with $\int f\,\d \mm=0$ such that 
\begin{equation*}
\int f^2\,\d \mm = \frac1{K} \int {|\nabla f|^2}\,\d \mm.
\end{equation*}
Then $f$ itself satisfies the properties in Lemma  \ref{lemma}.  In this case $f(r, y)=a_p r$ for a constant $a_p\in \R$.
\item [b)]  Logarithmic Sobolev inequality: $\Phi(x)=x\ln x$  for  $x\in \R^+$.
If  there is a non-negative   function $f\in \V$ with $\int f\,\d \mm=1$ such that 
\begin{equation*}
\int f\ln f\,\d \mm = \frac1{2K} \int \frac{|\nabla f|^2}{f}\,\d \mm.
\end{equation*}
Then  by Corollary \ref{coro:phi},   $\ln f$ attains the equality in the 2-Bakry-\'Emery inequality. In this case $f(r, y)=e^{a_l r-a_l^2/2K}$ for a constant $a_l \in \R$.
\end{itemize}

\end{proof}

\def\cprime{$'$}

\end{document}